\documentclass[12pt,reqno]{amsart}
\usepackage{amssymb}
\usepackage{amscd}
\usepackage{amsxtra}
\usepackage[mathscr]{eucal}
\usepackage[all]{xy}
\usepackage{amsmath}
\newtheorem{mydef}{Definition}
\newtheorem{mythm}{Theorem}
\newtheorem{mypro}{Proposition}
\newtheorem{mylem}{Lemma}

\usepackage[utf8]{inputenc}
\usepackage{tikz}
\newcommand{\unit}{\text{\textbf{1}}}

\newcommand{\Z}{{\mathbb Z}}

\theoremstyle{plain}
\numberwithin{equation}{section}

\theoremstyle{definition}

\theoremstyle{remark}

\newtheorem{conj}{Conjecture}[section]

\newcommand{\mcA}{\mathcal{A}}
\newcommand{\mcB}{\mathcal{B}}
\newcommand{\mcC}{\mathcal{C}}
\newcommand{\mcD}{\mathcal{D}}

\newcommand{\mcM}{\mathcal{M}}

\newcommand{\mcV}{\mathcal{V}}

\newcommand{\mbbB}{\mathbb{B}}
\newcommand{\mbbC}{\mathbb{C}}
\newcommand{\mbbD}{\mathbb{D}}

\newcommand{\mbbS}{\mathbb{S}}

\newcommand{\mbbZ}{\mathbb{Z}}
\author{Yang Qiu}
\address{Department of Mathematics, University of California, Santa Barbara, CA 93106, USA}
\email{yangqiu@math.ucsb.edu}

\author{Zhenghan Wang}
\address{Microsoft Station Q and Department of Mathematics, University of California, Santa Barbara, CA 93106, USA}
\email{zhenghwa@microsoft.com;zhenghwa@math.ucsb.edu}

\begin{document}

\title[Motion Group Representations]{Representations of Motion Groups of Links via Dimension Reduction of TQFTs}

\thanks{Z.W. is partially supported by NSF grants DMS 1411212 and  FRG-1664351.  The second author thanks C.-M. Jian, E. Samperton and K. Walker for related discussions and comments.}

\begin{abstract}
Motion groups of links in the three sphere $\mathbb{S}^3$ are generalizations of the braid groups, which are motion groups of points in the disk $\mathbb{D}^2$.  Representations of motion groups can be used to model statistics of extended objects such as closed strings in physics.  Each $1$-extended $(3+1)$-topological quantum field theory (TQFT) will provide representations of motion groups, but it is difficult to compute such representations explicitly in general.  In this paper, we compute  representations of the motion groups of links in $\mathbb{S}^3$ with generalized axes from Dijkgraaf-Witten (DW) TQFTs inspired by dimension reduction.  A succinct way to state our result is as a step toward a twisted generalization (Conjecture \ref{mainconjecture}) of a conjecture for DW theories of dimension reduction from $(3+1)$ to $(2+1)$: $\textrm{DW}^{3+1}_G \cong \oplus_{[g]\in [G]} \textrm{DW}^{2+1}_{C(g)}$, where the sum runs over all conjugacy classes $[g]\in [G]$ of $G$ and $C(g)$ the centralizer of any element $g\in [g]$.
We prove a version of Conjecture \ref{mainconjecture} for the mapping class groups of closed manifolds and the case of torus links labeled by pure fluxes.
\end{abstract}

%\subjclass[2000]{16W30}

%\date{\today}
\maketitle

\section{Introduction}

Topological quantum field theories (TQFTs) are used as the low energy effective description of topological phases of matter in physics, especially for anyon systems in two spacial dimensions.  A central part of an anyon model is the description of anyon statistics by the representation of braid groups (e.g. see \cite{rowell18}).  Braid groups are simply motion groups of points in the disk, therefore a natural generalization for statistics of extended objects in higher dimensions will be motion groups such as the motion groups of links in the three sphere $\mathbb{S}^3$ \cite{dahm62, goldsmith82}.

Representations of motion groups has been used to model statistics of extended objects such as closed strings in physics (e.g. see \cite{baez07,kadar17,levin14,rowell19}).  Each $1$-extended $(3+1)$-TQFT  will provide representations of the motion groups, but it is difficult to compute such representations explicitly in general.  In this paper, we compute representations of the motion groups of links in $\mathbb{S}^3$ with generalized axes from Dijkgraaf-Witten (DW) TQFTs inspired by dimension reduction.  A succinct way to state our results is a step towards a twisted generalization Conjecture \ref{mainconjecture} of a conjecture for DW theories using dimension reduction from $(3+1)$ to $(2+1)$: $\textrm{DW}^{3+1}_G \cong \oplus_{[g]\in [G]} \textrm{DW}^{2+1}_{C(g)}$, where the sum runs over all conjugacy classes $[g]$ of $G$ and $C(g)$ the centralizer of any element $g\in [g]$\footnote{Our main interest is for the twisted generalization of dimension reduction.  The direct product case of the conjecture as originally mentioned in the abstract has been proved since our paper appeared on the arxiv as outlined by an anonymous referee as below and the paper \cite{LW20}.  The referee pointed out: \lq\lq As the authors note, the (n, n+1)-part of DW theory assigns to a closed n-manifold the vector space of functions on the finite set $[\Pi(M), BG]$ of isomorphism classes of functors $\Pi(M) \rightarrow BG $(where $\Pi(M)$ is the fundamental groupoid of M and BG is the 1-groupoid with a single object and endomorphisms G), and to an n+1 bordism $W:M\rightarrow M'$ the \lq linearization' of the span of finite sets $[\Pi(M), BG] \leftarrow [\Pi(W), BG] \rightarrow [\Pi(M'), BG]$, 
(or more explicitly: the linear map with coefficients at a functor $F:\Pi(M) \rightarrow BG$ and a functor $F': \Pi(M')\rightarrow BG$ given by $\#${iso classes of functors $\Pi(W) \rightarrow BG$ such that the restrictions to $\Pi(M)$ and $\Pi(M')$ agree with F and F'}). Now as the authors note, there is a bijection of sets $[M \times S^1, BG] = [M, [S^1, BG]]$ = $\prod_{\textrm{conjugacy classes}} [M,BC(g)]$, inducing isomorphisms of vector spaces Functions $([M\times S^1, BG])$ = $\sum_{\textrm{conj. classes}} \textrm{Functions}([M, BC(g)]$. A similar computation shows that the bordism map arising from the span $([M\times S^1, BG]\leftarrow [W\times S^1, BG] \rightarrow [M' \times S^1, BG] )$= $\prod_{\textrm{conjugacy classes}} ( [M, BC(g)]\leftarrow [W, BC(g)]\rightarrow [M', BC(g)])$ is itself a direct sum of the appropriate linear maps. Hence, this gives an isomorphism of TQFTs between $DW^{3+1}_G$ and $\sum_{\textrm{conj. classes}} DW^{2+1}_{C(g)}$."}.  We prove a case of the main conjecture \ref{mainconjecture} for mapping class groups in Thm. 1.

Dimension reduction is a simple construction in relating quantum field theories of different dimensions.  Our interest in dimension reduction lies in the categorical context: if the input for an $(n+1)$-TQFT is a certain higher category, what are the resulting categories of the lower dimensional TQFTs from dimension reduction explicitly? If known, then representations of motion groups from the $(n+1)$-TQFT might be reconstructed from the lower dimensional ones.  In this paper, we study representations of the motion groups from $(3+1)$-DW TQFTs with such a goal in mind.  We conjecture that motion groups of links with generalized axes in $\mathbb{S}^3$ can be described using surface braid groups, and prove a version for the special case of torus links labeled by pure fluxes in Thm. 2.

Our main conjecture formulated as Conjecture \ref{mainconjecture} is that 
there is a map from labels of a $1$-extended $(3+1)$-TQFTs to those of $1$-extended $(2+1)$-TQFTs so that the representations of the motion groups of a link $L$ from the $(3+1)$-TQFT decompose as direct sums of the representations of some surface braid groups of the fiber surface $F$ if the link $L$ has a generalized axis $\gamma$ with fiber $F$.  This conjecture is our main focus and is a twisted version of dimension reduction for motion groups.  The standard dimension reduction neither touches on non-product fibrations nor motion group representations of links in $S^3$.

The link complement of a link with a generalized axis is a non-trivial fibration over the circle, hence the computation of the motion group representations does not follow from the product dimension reduction. Our conjecture is not specific to DW TQFTs and should hold for general TQFTs.  The difficulty lies in the right formulation of a correspondence of the label sets.  For $(3+1)$-DW TQFTs, the reduction of the labels for the torus boundary is related to the restriction of the homomorphisms from $\mbbZ^2$ to the meridian circle factor $\mbbZ$.

The content of our paper is as follows.  In Sec. \ref{basic}, we define $1$-extended TQFTs closely following the definition of Reshetikhin-Turaev TQFTs, recall dimension reduction, and outline the theory of motion groups.
In the end, we deduce presentations of the motion groups of the torus links from \cite{goldsmith82}.  In Sec. \ref{sec: DW}, we construct the $1$-extended truncation of DW TQFTs using colorings, find the label sets for the codim=$2$ excitations, and describe explicitly the representation spaces for motion groups of links labeled by pure fluxes.  Sec. \ref{DW:atiyah} is a straightforward generalization of the results in dimension $3$ from \cite{wakui92} to general $n$ dimensions. Sec. \ref{DW:extended} mainly rephrases results from \cite{morton07} in our framework.  We then prove a version of the conjecture for mapping class groups of closed manifolds.  In the last Sec. \ref{motiontorus}, we focus on the motion groups of torus links and prove a special case of our conjecture for torus links.

\subsection{Notations}\label{notation}

In this subsection, we collect some notation that are used throughout the paper.

1. $1$-extended $(n+1)$- or $(n+1,-1)$-TQFT $(Z,V, \mcC)$ for manifolds $(X^{n+1}, Y^{n}, \Sigma^{n-1})$ of dimensions $(n+1,n, n-1)$.

2. Representations vs homomorphisms: homomorphisms from fundamental groups to $G$ will be denoted as $\rho, \alpha,...,$ and the representations that they represent as $[\rho], [\alpha],...,$ since they are simply homomorphims up to conjugation.  Many constructions in this paper involve homomorphisms, not just representations.

3. Conjugacy class $[g]$ of a group $G$, centralizer $C_G(H)$ of a subgroup $H$ of $G$ in $G$ or $C_G(g)$ centralizer of an element $g\in G$ ($G$ will be dropped if no confusion arises).

4.  The vector space of an $n$-manifold $Y$ with boundary $\partial Y=\{\Sigma_i\}_{i=1}^m$ labeled by $\{([g_i,h_i],[\alpha_i])\}_{i=1}^m$ is denoted as $V_G(Y;([g_i,h_i],[\alpha_i]))$.

5. The torus links $TL_{(p,q)^n}$: for any pair of relatively prime natural numbers $p,q\geq 1$, $TL_{(p,q)^n}$ consists of $n$-copies of the $(p,q)$-torus knot.  If the torus is regarded as a folded square, then  $TL_{(p,q)^n}$ is just $n$ parallel copies of a slope $p/q$ curve in the square.

6. The $3$-sphere $\mbbS^3$ is identified with the union $\mathbb{C}\times \mbbS^1\cup \mbbS^1\times\mathbb{C}$ of two open solid tori, with the identifications $(re^{i\theta},e^{i\zeta})\sim(e^{i\theta},\frac{1}{r}e^{i\zeta})$ for $r>0$.  Let $x^{\prime},y^{\prime}$ be the circles $(0,e^{i\zeta}),(e^{i\theta},0)$, which are the cores of the two open solid tori.

7.  Let $\Xi_k$ be set of complex $k$-th roots of unity: $\Xi_k=\{\xi_k^l\}, l=0,1,...,k-1, \xi_k=e^{\frac{2\pi i}{k}}$, $R_t:\mathbb{C}\longrightarrow\mathbb{C}$ be the rotation by $2\pi t$ of the complex plane around the origin : $z \rightarrow e^{2\pi it}z$, and $\iota$ be an inclusion of a subset of $\mathbb{C}\times \mbbS^1$ into $\mbbS^3$.

8. $\pi_1(Y;B,A)$ for an $n$-manifold $Y$ with $m$ boundary components $\{\Sigma_i\}_{i=1}^m$ denotes the fundamental group of $Y$ together with $m$ inclusions of the $m$ fundamental groups $\{\pi_1(\Sigma_i,b_i)\}_{i=1}^m$ of the boundaries to $\pi_1(Y,b)$, where $B=\{b,b_1,..,b_m)\}$ are based points for $Y$ and the $i$-th boundary component $\Sigma_i$, respectively, and $A=\{A_i\}_{i=1}^m$ are $m$ arcs connecting $b$ to $b_i$, respectively.  We choose $A_i$ to be disjoint and embeded if $n\geq 3$.

9.  A homomorphism from $\pi_1(Y;B,A)$ to a group $G$ is a triple $(\rho, \{\alpha_i\}_{i=1}^m, \{g_i\}_{i=1}^m)$ of $\rho: \pi_1(Y,b)\rightarrow G, \alpha_i: \pi_1(\Sigma_i,b_i)\rightarrow G, g_i\in G$ such that for each $i$, $\rho=g_i\cdot \alpha_i\cdot g_i^{-1}$ when restricted to $\pi_1(\Sigma_i,b_i)$. 

10.  The motion group of an oriented submanifold $N$ in the interior of an oriented ambient manifold $M$ is denoted by $\mathcal{M}(N\subset M)$.  The orientation preserving diffeomorphism group of $M$ that fixes $N$ as an oriented submanifold is denoted as $\mathcal{H}^+(M;N)$.  In the case of a surface braid group, we also use the notation $\mathcal{B}(F;P)$ for a surface $F$ and a finite collection of points in the interior of $F$.

\section{Extended TQFTs, Dimension Reduction, and Motion Groups}\label{basic}

Atiyah type TQFTs do not necessarily lead to representations of motion groups, but fully extended ones do.  In this section, we formulate $1$-extended $(n+1)$-TQFTs as generalizations of Reshetikhin-Turaev TQFTs, which always give rise to representations of motion groups of codimension=$1$ submanifolds of spacial $n$-manifolds.  We are especially interested in three spacial dimensions, so motion groups of links in three manifolds.

\subsection{$(n+1,-k)$-TQFTs}

A $k$-extended $(n+1)$- or $(n+1,-k)$-TQFT is one that is extended from $(n+1)$-manifolds all the way to $(n-k)$-manifolds. When $k=0$, they are the Atiyah-type TQFTs, while $k=n$ the fully-extended ones. Formally, they could be defined as monoidal functors between appropriate higher categories. For explicit calculations, a more elementary approach is preferred, and furthermore, we will allow theories with framing anomaly. Our focus is on $k=1$, so we will define $1$-extended TQFTs using explicit axioms similar to the Reshetikhin-Turaev $(2+1)$-TQFTs, which are $1$-extended with framing anomaly in general.

\begin{mydef}

A $1$-extended $(n+1)$-TQFT or $(n+1,-1)$-TQFT is a triple $(Z,V,\mcC)$, where $(Z,V)$ is a projectively\footnote{Our axioms are straightforward generalizations of the axioms in \cite{rowell18} for $(2+1)$ dimensions.} symmetric monoidal functor from the category $\textrm{Bord}^{n+1}_n$ of $(n+1)$- and $n$-manifolds to the category $\mcV{ec}$ of finitely dimensional complex vector spaces, and in addition to this projective Atiyah type TQFT, an assignment of a semi-simple finite category $\mcC(\Sigma)$ to each oriented closed $(n-1)$-manifold $\Sigma$ and a finitely dimensional vector space $V(Y;\{X_l\})$ to each oriented $n$-manifold $Y$ with parameterized and labeled boundary components by $X_l\in \Pi_\mcC(\partial Y)$\footnote{$\Pi_{\mcC}$ is a complete set of the simple representatives of a category $\mcC$. Each connected boundary component of $Y$ is labeled by an object in $\Pi_{\mcC}$.  We use $\simeq$ for vector space isomorphism and $\cong$ for functorial isomorphism.} such that the following axioms hold:

\begin{enumerate}

\item \textrm{Empty manifold axiom:}
$V(\emptyset)=1, \mbbC, \textrm{or}\; {\mcV}ec$ if $\emptyset$ is regarded as a manifold of dimension=$n+1,n,n-1$, respectively.

\item \textrm{Disk axiom:}

$V(D^n;X_i)\cong \begin{cases}
\mbbC, & X_i=\unit,\\
0, & \textrm{Otherwise}, \end{cases}$
\quad \quad where $D^n$ is an $n$-disk and $\unit$ the tensor unit.

\item \textrm{Cylinder axiom:}

\[V(\mcA;X_i,X_j)\simeq \begin{cases} \mbbC & \textrm{if $X_i\simeq X_j^{*}$},\\ 0 & \textrm{otherwise}, \end{cases}\]
where $\mcA$ is the cylinder $\mbbS^{n-1} \times I$, and $X_i,X_j\in \Pi_{\mcC}(\mbbS^{n-1})$.  Furthermore, $V(\mcA;X_i,X_j)\cong \mbbC$ if  $X_i\cong X_j^{*}$.  The notations $\simeq$ and $\cong$ denote isomorphism as vector spaces and functorial isomorphism, respectively.  The difference is necessary due to Frobenius-Schur indicators of labels\footnote{A self-dual simple object type in a spherical fusion category has a Frobenius-Schur (FS) indicator \cite{Wang}, which creates a subtlety that makes the original formulation of gluing axioms in \cite{Walker1991} not general enough to cover theories with non-trivial FS indicators such as $SU(2)_k$.  The problem is traced back to the cylinder axiom.  In the case with non-trivial FS indicators, the identity functor on self-dual objects is not functorial, which leads to inconsistency.   

The inconsistency arose because manifolds with non-empty boundaries have the freedom to absorb a cylindrical neighborhood of the boundary.  The isomorphism of the vector spaces from a TQFT requires functoriality when tensoring the 1-dimension vector space from a cylinder.  In the case with non-trivial FS indicators, the $-1$ would cause inconsistency.  The standard solution as in Turaev’s book \cite{Turaev} is to expand the label set with signed simple objects, hence a (2+1)-TQFT is really a (1+1+1)-TQFT as in graphical calculus.}.

\item \textrm{Disjoint union axiom:}

$V(Y_1 \sqcup Y_2;X_{\ell_1}\sqcup X_{\ell_2})\cong V(Y_1;X_{\ell_1})\otimes
V(Y_2;X_{\ell_2})$.
\\
The isomorphisms are associative, and
compatible with the mapping class group projective actions $V(f): V(Y)\rightarrow V(Y)$ for $f: Y\rightarrow Y$.

\item \textrm{Duality axiom:}

$V(-Y;X_{\ell})\cong V(Y;X_{\ell})^*$, where $-Y$ is $Y$ with the opposite orientation.
The isomorphisms are compatible with mapping class group projective actions, orientation reversal, and the disjoint union axiom as follows:

(i): The isomorphisms $V(Y) \rightarrow V(-Y)^*$ and $V(-Y)\rightarrow V(Y)^*$ are mutually adjoint.

(ii): Given $f: (Y_1;X_{\ell_1})\rightarrow (Y_2;X_{\ell_2})$ 
let $\bar{f}:(-Y_1; X_{\ell_1}^*)\rightarrow (-Y_2;X_{\ell_2}^*)$ be the induced reversed orientation map, we have
$\langle x,y\rangle =\langle V(f)x,V(\bar{f})y\rangle $, where $x\in V(Y_1;X_{\ell_1})$, $y\in V(-Y_1;X_{\ell_1}^*)$.

(iii): $\langle \alpha_1\otimes \alpha_2, \beta_1\otimes \beta_2\rangle =\langle \alpha_1, \beta_1\rangle \langle \alpha_2,\beta_2\rangle$ whenever
\begin{align*}
\alpha_1\otimes \alpha_2 &\in V(Y_1\sqcup Y_2)\cong V(Y_1)\otimes V(Y_2),\\
\beta_1\otimes \beta_2 &\in V(-Y_1\sqcup -Y_2)\cong V(-Y_1)\otimes V(-Y_2).
      \end{align*}

\item \textrm{Gluing axiom:}
Let $Y_{\mathrm{gl}}$ be the surface obtained from
gluing two boundary components $\Sigma$ of a surface $Y$ by a diffeomorphism which is isotopic to the orientation reversing \lq\lq identity"\footnote{The gluing operation should be regarded as the inverse of cutting the manifold $Y_{\mathrm{gl}}$ along the glued boundary component.}.
Then
\[V(Y_{\mathrm{gl}})\cong\bigoplus_{X_i\in \Pi_\mcC(\Sigma)} V(Y;(X_i,X_i^*)).\]
 The isomorphism is associative and
compatible with mapping class group projective actions.

Moreover, the isomorphism is compatible with duality as follows:
Let
\begin{align*}
\bigoplus_{j\in \Pi_{\mcC}}\alpha_j &\in V(Y_{\mathrm{gl}};X_{\ell})\cong \bigoplus_{j\in \Pi_{\mcC}}V(Y;X_{\ell},(X_j,X_j^*)),\\
\bigoplus_{j\in \Pi_{\mcC}}\beta_j &\in V(-Y_{\mathrm{gl}};X_{\ell}^*)\cong \bigoplus_{j\in \Pi_{\mcC}} V(-Y;X_{\ell}^*,(X_j,X_j^*)).
\end{align*}
Then there is a nonzero real number $s_{j}$ for each label $j$ such that
\[\biggl\langle \bigoplus_{j\in \Pi_{\mcC}}\alpha_j, \bigoplus_{j\in \Pi_{\mcC}}\beta_j\biggr\rangle =\sum_{j\in \Pi_{\mcC}}
s_j\langle \alpha_j,\beta_j\rangle .\]

\end{enumerate}

\end{mydef}

In Sec. \ref{DW:extended}, we explicitly describe how to view $(3+1)$-DW TQFTs as $1$-extended ones in the sense above.  
All state-sum TQFTs are fully extended so they should be examples of $1$-extended TQFTs.  More interesting examples are Reshetikhin-Turaev TQFTs, which are $1$-extended examples that are not fully extended.  We are not aware of any $(3+1)$-TQFTs which are $1$-extended, but not fully extended.  Potentially such examples can be constructed analogous to the construction of Reshetikhin-Turaev TQFTs as state-sum $(3+1)$-TQFTs \cite{walker13}.

In a $1$-extended $(n+1)$-TQFT, each oriented closed $(n-1)$-manifold $\Sigma$ is assigned a semi-simple finite category $\mcC(\Sigma)$.  In Crane-Yetter TQFTs from ribbon fusion categories, all such categories assigned to $\Sigma$'s should come from the same input ribbon fusion category as follows: given an oriented closed $2$-manifold $\Sigma$ with a base point, and an input ribbon fusion category $\mcB$ for the Crane-Yetter TQFT, a picture cylindrical category $\mcA(\Sigma \times I)$ can be constructed: the objects of $\mcA(\Sigma \times I)$ are finitely many signed framed points colored by objects of $\mcB$, and morphisms are ribbon graphs between objects colored by morphisms of $\mcB$ up to framed ribbon isotopies relative to the base point in $\Sigma \times I$.   The representation category $\mcC(\Sigma)$ of $\mcA(\Sigma \times I)$ is a semi-simple category that is assigned to $\Sigma$. The label set $\Pi_\mcB(\Sigma)$ is then the set of irreducible representations of the picture cylindrical category $\mcA(\Sigma \times I)$.  For the $1$-truncations of fully extended $(3+1)$-TQFTs from spherical fusion $2$-category as in \cite{reutter18}, all the semi-simple finite categories associated to different surfaces should arise as some doubles.

\subsection{Dimension Reduction}\label{sec:DR}

Dimension reduction (DR) is a simple construction that is widely used in physics that produces an $((n-1)+1)$-TQFT from an $(n+1)$-TQFT for Atiyah type TQFTs. But DR is not always possible in our set-up because in our definition, the vector space associated to the sphere $\mbbS^{n}$ is always $\mbbC$ by the gluing axiom.  The condition $V(\mbbS^n) \cong \mbbC$ is a stability condition in topological physics.  Since the only stable $(1+1)$-TQFT in our sense is trivial, non-trivial $(2+1)$-TQFT cannot be reduced to $(1+1)$ via DR.  Since DR can always be done within Atiyah type TQFTs, hence we discuss DR for only TQFTs regarded as Atiyah type ones. 

\begin{mydef}
Let $(Z,V)$ be an $(n+1)$-TQFT. The resulting $((n-1)+1)$-TQFT $(Z_{DR}, V_{DR})$ defined by 
$$V_{DR}(Y)=V(Y\times S^1)
$$
$$Z_{DR}(X:Y_1\longrightarrow Y_2)=Z(X\times S^1:Y_1\times S^1\longrightarrow Y_2\times S^1)
$$
is called the DR of $(Z,V)$.
\end{mydef}

\begin{mypro}
The fusion algebra $V(T^2)$ of any $1$-extended $(2+1)$-TQFT is a Frobenius algebra, and the reduction of a $(2+1)$-TQFT is the $(1+1)$-TQFT given by the Frobenius algebra $V(T^2)$.
\end{mypro}

\begin{proof}
Each $1$-extended $(2+1)$-TQFT is associated with a modular tensor category in our definition and the modular $S=(s_{ab})_{a,b\in L}$-matrix diagonalizes the fusion rules. It follows that in the basis given by $e_a=\sum_{b\in L}s_{ab}b$, the fusion algebra $V(T^2)$ becomes the algebra $A=\sum_{a\in L}\mbbC[e_a]$. It follows that $A$ is a Frobenius algebra of the direct sum of $1$-dimensional Frobenius algebras.
\end{proof}

The proposition illustrates that dimension reduction in general does not preserve the condition that the dimension of the vector spaces from a TQFT for spheres is $1$-dimensional, which is regarded as a stability condition in topological phases of matter.

\subsection{Motion groups}

We first recall some basic notions for motion groups following \cite{dahm62, goldsmith82}, then describe motion groups of links with generalized axes.  In the end, we derive presentations of motion groups of torus links from [Thm 8.7 \cite{goldsmith82}].

\subsubsection{Basic notions}
We work in the smooth category in this section.

Let $N$ be an oriented compact non-empty sub-manifold in the interior of an oriented manifold $M$. A motion $H(x,t)$ of $N$ inside $M$ is an ambient isotopy $H:M \times I \to M \times I$ of $M$ such that $H(x,0)=\textrm{id}_M$ and $H(x,1)(N)=N$ as an oriented sub-manifold, 
The orientation preserving diffeomorphism $H(x,1): M \backslash N \to M \backslash N$ preserves the complement of $N$ in $M$.  A motion $H(x,t)$ can be equally regarded as a path $h_t=H(x,t)$ in the orientation preserving diffeomorphism group $\mathcal{H}^+(M)$ of $M$ based at the identity. Note that isotopy classes of $\mathcal{H}^+(M)$ is the mapping class group of $M$, while motion groups concern only the transformation of non-empty sub-manifolds in $M$ by diffeomorphisms isotopic to the identity.

The composition $H_2 \circ H_1$ of two motions $H_1$ and $H_2$ is defined by:
$$H_2 \circ H_1 = \begin{cases} H_1(x,2t),  & 0 \le t \le 1  \\
H_2(x,2t-1) \circ h_1, & \frac{1}{2} \le t \le 1 \end{cases}.$$
Inversion of a motion is defined by $H^{-1}(x,t)=H(x,1-t)\circ h_1^{-1}$. 

A motion $H(x,t)$ is stationary if $H(x,t)(N)=N$ as a set for all $t \in (0,1)$.  While $H(x,1)(N)=N$ as oriented manifolds, for some $t\in (0,1)$ $H(x,t)$ could be orientation reversing when restricted to the sub-manifold $N$.

Two motions $H_1$, $H_2$ of $N \subset M$ are equivalent if $H_2^{-1} \circ H_1$ is homotopic to a stationary motion.

\begin{mydef}[Motion group]\footnote{In \cite{goldsmith82}, $H(x,1)$ is allowed to be orientation reversing when restricted to $N$, so our motion groups are the oriented ones there.} 
Given oriented manifolds $N \subset M$ as above, then the motion group $\mathcal{M}(N\subset M)$ is the group of motions modulo stationary ones.
\end{mydef}

If $N$ is a point $b$ of $M$, then $\mathcal{M}(N \subset M)$ is the fundamental group $\pi_1(M,b)$ of $M$.  If $N$ consists of $n$ distinct points in the disk $\mbbD^2$, then $\mathcal{M}(N\subset M)$ is the braid group $\mbbB_n$.

When $M=\mathbb{S}^3$, the following exact sequence holds:
%$$
\begin{equation}\label{equ: Dahm}
\mathbb{Z}_2\longrightarrow\mathcal{M}(\sqcup_{i=1}^n N_i\subset M)\stackrel{\partial}{\longrightarrow}\mathcal{H}^+(M;N_1,...,N_n)\longrightarrow 1,
\end{equation}
%$$
where $\partial$ is the Dahm homomorphism [Corollary 1.13 \cite{goldsmith82}].
\subsubsection{Links with generalized axes}

The necklace link in Fig. \ref{fig:my_labelnecklace} is an example of a link with an axis: the unknot train track that the unlink winds around is the axis.  In general, the unlink can be replaced by any braid closure $L$, therefore any link $L$ is a link with an axis.  It is obvious there is a connection between the motion group of $L\cup \gamma$---the link $L$ together with the axis $\gamma$---and the motion group of the intersection points of the link $L$ with the spanning disk $\mbbD^2$ of the unknot $\gamma$ in $\mbbD^2$.  But the relation could be complicated, and also our interest is not on a relation between the motion group of $L\cup \gamma$ and the braid group, rather the motion group of the link $L$ itself and some braid group.  In \cite{goldsmith82}, Goldsmith proved that if the unknot is generalized to a fibered knot $\gamma$ as a generalized axis, then in some cases the motion group of the link $L$ itself is the same as the motion group of $L\cup \gamma$.  One such case is the torus link $TL_{(p,q)^n}$ of $n$-copies of the $(p,q)$ torus knot, which is our focus in this paper.  This is not true in general as the motion group of the unlink---the loop braid group---is different from that of the necklace link---the annulus braid group.

A generalized axis $\gamma$ for some link $L$ is a non-trivial fibered knot $\gamma$ with a fixed fibration $\pi : \mbbS^3 \backslash \gamma  \rightarrow \mbbS^1$. Let $F=\pi^{-1}(1),1\in \mbbS^1$, with closure $\bar{F}$ in $\mbbS^3$, then $\partial \bar{F}=\gamma$.
The fibration is given by a surjective map $f: F\times I \rightarrow \mbbS^3\backslash \gamma, I=[0,1],$ such that 

\begin{enumerate}
    \item $f: F\times (0,1) \rightarrow \mbbS^3\backslash \gamma$ is a diffeomorphism,
    \item $f$ extends to the closure $\bar{F}\times I$ such that $f_0=$id and $f_1$ is a diffeomorhpism of $F$ with compact support, which is called the monodromy.
    \item $f: \gamma \times (0,1)\rightarrow \gamma$ is a projection.
\end{enumerate}

A link $L$ has a generalized axis $\gamma$ if the link $L$ in $\mbbS^3$ is in a braid position with respect to the fixed fibration $F\longrightarrow \mbbS^3 \backslash \gamma\longrightarrow \mbbS^1$ \cite{goldsmith82}, i.e. if a link component is parameterized, then as the parameter increases, so will be the fibration parameter in $I$ of $F\times I$ periodically. Set $P=F\cap L$ and $\phi :(F,P)\longrightarrow(F,P)$ the monodromy. There exists an exact sequence:
%$$
\begin{equation}\label{equ:fiber}
1\longrightarrow<[\phi],[\tau]>\longrightarrow\mathcal{H}_{\phi}(F;P)\stackrel{eJ}{\longrightarrow}\mathcal{H}^+(\mathbb{S}^3;L,\gamma)\longrightarrow\mathbb{Z}_2,
\end{equation}
%$$
where $\mathcal{H}_{\phi}$ is the centralizer of $[\phi]\in\mathcal{H}^+(F;P)=\mathcal{B}(F;P)$, and $[\tau]$ is the central Dehn twist on a push-off of the axis $\gamma$ to a collar of $F$, and $<[\phi],[\tau]>$ is the subgroup generated by $[\phi],[\tau]$  [Thm 5.26 \cite{goldsmith82}].

\subsubsection{Presentations of the motion groups of torus links}

The motion groups of links are interesting generalizations of the braid groups.  Only a few of them are being investigated recently, partially due to the application to statistics of loop excitations in physics.  The motion groups of the unlinks, which are loop braid groups \cite{baez07, kadar17}, and the motion groups of the necklace links \cite{levin14,rowell19}, are the main focus.  Our interest is on the motion groups of the torus link $TL_{(p,q)^n}$---$n$ parallel copies of the $(p,q)$-torus knot.  

In \cite{goldsmith82}, Goldsmith obtained a presentation of the motion groups of the torus link $TL_{(p,q)^n}$ using the theory of motion groups of links with a generalized axis.  Torus links $TL_{(p,q)^n}$ are examples of links $L$ such that the motion groups of $L\cup \gamma$---the link plus the axis---is the same as the link $L$ itself (without the axis).  We observe that the presentation in [Thm 8.7 \cite{goldsmith82}] implies that actually there are only three families of motions groups of the torus links $TL_{(p,q)^n}$: $(p,q)=(1,1), p+q=$odd, or $p+q=$even indexed by $n$.
Hence as abstract groups, when $(p,q)\neq (1,1)$, the motion groups of $TL_{(p,q)^n}$ depend only on the parity of $p+q$.  

Presentations of the three families of motions groups indexed by $n$ for $TL_{(p,q)^n}, p,q>1$ are as follows.

\begin{mypro}\label{presentation}

\begin{enumerate}
    \item If $p+q$ with $p,q\geq 1$ is odd, then the odd motion groups $\mcM TL_{n,-}$ of the torus links has a presentation:
    $$<\sigma_1,\cdots, \sigma_{n-1},r_1,\cdots, r_n|\{\sigma_j\}_{j=1}^{n-1}\;\textrm{satisfy the braid relations},\;
    $$
    $$r_1\cdots r_n=1, r_ir_k=r_kr_i, r_i\sigma_j=\sigma_jr_i, 1\leq i,k\leq n, j\neq i-1>,$$
    \item If $p+q$ is even but not $p=q=1$, then the even motion groups $\mcM TL_{n,+}$ of the torus links has a presentation:
    $$<\sigma_1,\cdots, \sigma_{n-1},r_1,\cdots, r_n, r_{2\pi}|\{\sigma_j\}_{j=1}^{n-1}\;\textrm{satisfy the braid relations},\;$$
    $$r_1\cdots r_n=1, r_ir_k=r_kr_i, r_i\sigma_j=\sigma_jr_i, j\neq i-1, r_1\cdots r_n=r_{2\pi}, r_{2\pi}^2=1>,$$
    \item If $p=q=1$, then the motion groups $\mcM H_{n}$ of the $n$-Hopf links has a presentation:
    $$<\sigma_1,\cdots, \sigma_{n-1},r_1,\cdots, r_n|\{\sigma_j\}_{j=1}^{n-1}\;\textrm{satisfy the braid relations},\;$$
    $$r_1\cdots r_n=1, r_ir_k=r_kr_i, r_i\sigma_j=\sigma_jr_i, j\neq i-1, r_1=r_n=1>.$$
\end{enumerate}

\end{mypro}

These presentations are derived from [Thm 8.7 \cite{goldsmith82}].  First our $r_i$'s are the $\rho_i$'s there.  Secondly the generator $[f]$ there does not exist in our oriented motion group as it is an orientation reversing diffeomorhpism. There is a typo in relation $7$---one $p$ is a $q$, and then our presentations follow from the existence of integers $u,v$ such that $pu-qv=1$. 

\section{Dijkgraaf-Witten TQFTs}\label{sec: DW}

Given a finite group $G$, the untwisted DW $(n+1)$-TQFTs based on $G$ for any $n\geq 1$ are among the best understood examples of fully extended TQFTs.  In this section, we provide an elementary formulation based on a generalization of the combinatorial construction in \cite{wakui92} from coloring triangulations and the extension to manifolds with boundaries in \cite{morton07}.  Our goal is to set-up notation, and to describe the label sets for the $1$-extended truncation and the vector spaces as representations of the motion groups of links in the three sphere explicitly.

\subsection{DW TQFTs as Atiyah type}\label{DW:atiyah}

Let $M$ be an oriented compact triangulated $m$-manifold with boundary $\partial M$. Let $v$ be the number of the vertices of $M$ and $\partial v$ the number of the vertices on $\partial M$.  We use both $\sharp S$ and $|S|$ to denote the number of elements in a set $S$.

\begin{mydef}
Given a finite group $G$, a coloring $\varphi=[c]$ of the triangulated manifold $M$ is an equivalence class of assignments of an orientation $\pm$ and a group element $g\in G$ to each edge of $M$:
$$c:\{\text{edges of }M\}\longrightarrow (\pm,G)
$$ that for any oriented triangle, the colors of the three edges satisfy:
$1\stackrel{g}{\longrightarrow} 2, 2\stackrel{h}{\longrightarrow}3$, then the edge  $1\stackrel{gh}{\longrightarrow} 3$.  The equivalence class of a coloring is generated by the relation that 
if an oriented edge is colored by $g$, then the edge with the opposite orientation is colored by $g^{-1}$.

\end{mydef}

For simplicity, we suppose that $M$ is connected. A coloring $\varphi$ essentially defines a flat principle $G$-bundle on $M$. After a vertex $x$ is chosen as a base-point of $M$, the holonomy representation of $\varphi$ defines a homomorphism $\varphi_*:\pi_1(M,x)\longrightarrow G$, which is determined in the following combinatorial way. Any loop $\alpha$ at $x$ can be homotopic to a loop consisting of some edges of the triangulation of $M$. Then $\varphi_*(\alpha)$ is defined to be the product of the elements coloring the edges following the direction of $\alpha$.

Conversely, each homomorphism $\rho:\pi_1(M,x)\longrightarrow G$ gives rise to colorings as follows: pick any maximal tree $T$ of the $1$-skeleton $M^{(1)}$ of $M$ and retract the tree $T$ to the base point $x$, then the $1$-skeleton $M^{(1)}$ contracts to a bouquet of circles.  Each circle in the retraction receives a group element from $\rho$, then coloring every edge of $T$ by the group unit and each circle according to the image of $\rho$ leads to a desired coloring.

The set of all colorings of $M$ will be denoted by Col$(M)$ and the set of all colorings of $M$ with restriction on $\partial M$ being a particulr color $\tau$ by Col$(M,\tau)$.

\begin{mydef}
For any $M$ as above and $\tau$ a coloring of $\partial M$, the partition function $Z(M,\tau)$ is defined as
$$Z(M,\tau)=|G|^{\frac{\partial v}{2}-v}\sharp\text{Col}(M,\tau).
$$
\end{mydef}

\begin{mypro}\label{DW:Prop2}
If a triangulation and a coloring $\tau$ of $\partial M$ are fixed, then $Z(M,\tau)$ does not depend on the extending triangulations of $M$.
\end{mypro}

Prop. \ref{DW:Prop2} follows directly from Lemmas \ref{DW:Lem1} and \ref{DW:Lem2} below.

\begin{mylem}\label{DW:Lem1}
$$\sharp\text{Col}(M)=|G|^{v-1}\sharp\text{Hom}(\pi_1(M,x),G)
$$
where $x$ is any base-point of $M$.
\end{mylem}
\begin{proof}
The triangulation of $M$ leads to a chart atlas for $M$ such that each vertex indexes an open ball and each edge means that the two open balls for the two end points intersects in an open ball. Thus for any $G$-principle bundle over $M$, a trivialization on this atlas with transition functions as elements of $G$ coloring edges defines a coloring of $M$. Since $G$ is a finite group, the classifying space of $G$ is $K(G,1)$. It follows that for any $f\in$ Hom$(\pi_1(M,x),G)$, there exists a $G$-principle bundle on $M$ whose holonomy representation is $f$, which defines a coloring of $M$.

To show that for any $f\in$ Hom$(\pi_1(M,x),G)$, there is a canonical way to  color the vertices resulting all the colorings whose holonomy representation is $f$, we choose a vertex $x$ as the base point and use $G$ to color all the vertices except the base point $x$. Let $\varphi$ be a coloring of $M$ realizing $f$. By each coloring $\phi$ for vertices, we can change $\varphi$ to $\varphi_{\phi}$ as follows:  $a(h)\stackrel{g}{\longrightarrow}b(k)$, where $g$ colors the edge $ab$ and $h$ colors the vertex $a$, $k$ colors the vertex $b$, is changed to be $a\stackrel{h^{-1}gk}{\longrightarrow}b$. A straightforward check shows that  $\varphi_{\phi}$ is a still coloring of $M$ whose holonomy representation is $f$. Pairwise different $\phi$ result in pairwise different $\varphi_{\phi}$.
To show that any coloring whose holonomy representation is $f$ can be obtained in this way, note that since the group unit $1$ is used to color the base point, we can find the element coloring each vertex from the vertices near the base point to the ones far from the base point by the elements coloring edges.  Finally, there are $|G|^{v-1}$ colorings for vertices.
\end{proof}

\begin{mylem}\label{DW:Lem2}
Let $B_1,...,B_l$ be the components of $\partial M$ and $b\in M,b_i\in B_i$ be the base points of $M,B_i$, respectively.  A coloring $\tau$ of $\partial M$ decomposes as: $\tau=\bigsqcup_{i=1}^{l}\tau_i$. Choose paths $\gamma_1,...\gamma_l$ connecting $b$ to $b_1,...,b_l$, we obtain 
$\sharp\text{Col}(M,\tau)=$
\begin{align*}
|G|^{v-\partial v}\prod_{i=1}^l|C_G(Im(\tau_i)_{\ast})|\sharp\{f\in\text{Hom}(\pi_1(M,b),G)|f(\gamma_i)_{\ast}(\iota_i)_{\ast}\sim(\tau_i)_{\ast}\}
\end{align*}
where $\iota_i:B_i\longrightarrow \partial M$ is the inclusion.  The induced homomorphism $(\gamma_i)_*:\pi_1(M,b_i)\longrightarrow\pi_1(M,b)$ comes from $\gamma_i$ and $\sim$ means conjugation as two homomorphisms.  Here $\varphi_*$ denotes the corresponding group homomorphism for a coloring $\varphi$.
\end{mylem}

\begin{proof}
Any $\varphi\in\text{Col}(M,\tau)$ leads to  $(\varphi)_*(\gamma_i)_{\ast}(\iota_i)_{\ast}=g_i(\tau_i)_{\ast}g^{-1},$ where $g_i$ colors the path $\gamma_i$. Thus  $\varphi_*(\gamma_i)_{\ast}(\iota_i)_{\ast}\sim(\tau_i)_{\ast}$. 

When $$\{f\in\text{Hom}(\pi_1(M,x),G)|f(\gamma_i)_{\ast}(\iota_i)_{\ast}\sim(\tau_i)_{\ast}\}$$ is empty, there are no colorings in Col$(M,\tau)$.
When there exists $f\in\{f\in\text{Hom}(\pi_1(M,x),G)|f(\gamma_i)_{\ast}(\iota_i)_{\ast}\sim(\tau_i)_{\ast}\}$, there exists a coloring $\varphi$ of $M$ whose holonomy representations on $M,B_i$ are $f,h_i(\tau_i)_*h_i^{-1}$, respectively. By coloring the base point $x_i$ of $B_i$ with $h_i$, the method in Lemma \ref{DW:Lem1} can be used to modify $\varphi$ so that its holonomy representations on $M,B_i$ are $f,(\tau_i)_{\ast}$, respectively. Then we follow the same method by coloring vertices of $\partial M$ to find a coloring of $M$ whose restriction on $\partial M$ is $\tau$, which is still denoted by $\varphi$. By the same argument in Lemma \ref{DW:Lem1}, we see that all the colorings in Col$(M,\tau)$ whose holonomy representation is $f$ is a modification of $\varphi$ by coloring the vertices.

Next we count the colorings in Col$(M,\tau)$ whose holonomy representation is $f$. Since the restriction of colorings on $B_i$ is fixed, only the elements in $C_G(Im(\tau_i)_*)$ can be used to color the base point $x_i$ of $B_i$ and $1$ to color the other vertices in $B_i$. There are $|G|^{v-\partial}\prod_{i=1}^l|C_G({Im(\tau_i)_*})|$ colorings for vertices of $M$, and the proof is completed.
\end{proof}

To construct the $(n+1)$-DW TQFT $(Z_G,V_G)$,  for any oriented closed triangulated $n$-manifold $Y$, we first define a vector space $\widetilde{V}(Y)$ to be $\mathbb{C}[\text{Col}(Y)]$---the vector space spanned by colorings.
Given a bordism $\zeta=(X,Y_1,Y_2,f_1,f_2)$, where $X$ is an oriented compact $(n+1)$-manifold with boundary $\partial X$, and $Y_1,Y_2$ are oriented closed triangulated $n$-manifolds with an orientation-preserving diffeomorphism:
$$f_1\sqcup f_2:Y_1\sqcup {-Y_2}\longrightarrow\partial X.
$$

Let $Z(\zeta):\widetilde{V}(Y_1)\longrightarrow \widetilde{V}(Y_2)$ be
$$Z(\zeta)(\tau)=\sum_{\mu\in\text{Col}(Y_2)}Z(X,\tau\sqcup\mu)\mu,
$$
where $\tau\in\text{Col}(Y_1)$ and $\tau\sqcup\mu$ is the coloring of $\partial X$ induced by $f_1\sqcup f_2$. $Z(\zeta)$ is well-defined by Prop. \ref{DW:Prop2}.

The proof of the following is straightforward.
\begin{mylem}
$Z(\zeta\eta)=Z(\zeta)Z(\eta)$ for any two bordisms $\zeta,\eta$.
\end{mylem}

Thus $Z(Id_{Y})$, $Id_{Y}=(Y\times I,Y\times{0},Y\times{1},id,id))$, is an idempotent from $\widetilde{V}(Y)$ to itself. The Hilbert space $V_G(Y)$ is then defined to be the image of $Z(Id_Y)$. The restriction of $Z(\zeta)$, $\zeta=(X,Y_1,Y_2,f_1,f_2))$ to $Z(Id_{Y_i}),i=1,2$, defines a map from $V_G(Y_1)$ to $V_G(Y_2)$. Furthermore, equivalent bordisms define the same linear map, in other words, $Z(\zeta)$ is defined on the equivalence class of bordisms. 

The DW-TQFT $(Z_G,V_G)$ defined above leads to a natural representation of the mapping class group $\mcM(Y)$ on the Hilbert space $V_G(Y)$ as follows:
any $[f]\in \mcM(Y)$ gives rise to a bordism $\zeta_f=(Y\times I,Y\times{0},Y\times{1},f,id)$. Then the linear transformation $Z(\zeta_f)$ from $V_G(Y)$ to itself defines a representation  $\rho:\mcM(Y)\longrightarrow GL(V(Y))$ by $\rho([f])=Z(\zeta_f)$.

\begin{mypro}
There exists a basis $\{e_k\}$ for $V_G(Y)$ in $1-1$ correspondence with Hom$(\pi_1(Y),G)/\sim$ such that the action of $[f]\in \mcM(Y)$ on $V_G(Y)$ is the permutation of $\{e_k\}$ by pre-composing $f: \pi_1(Y) \rightarrow \pi_1(Y)$. Thus, the following diagram commutes:
$$\xymatrix{
    \mcM(Y)\ar[rr]\ar[dr]^{\rho} & & S_N \ar[dl]\\
   & GL(V_G(Y)), &
   }
$$
where $N=$dim$(V_G(Y))$.  

The representation is always reducible as the conjugacy class of the trivial homomorphism is fixed by $\mcM(Y)$ by pre-composition. Thus one element in $\{e_k\}$ is always fixed by $\mcM(Y)$. Hence $N$ can be reduced to $N-1$=dim$(V_G(Y))$-1.
\end{mypro}
\begin{proof}
Let $v$ be the number of vertices of $Y$.
By Lemma \ref{DW:Lem1}, dim$V_G(Y)=|G|^{v-1}\sharp\text{Hom}(\pi_1(Y),G)$ and a basis for $V_G(Y)$ is given as follows.
$$\tau^{i,j}_k\in\text{Col}(Y),
$$
where $k$ indexes the conjugacy class of $(\tau^{i,j}_k)_*$ in Hom$(\pi_1(Y),G)/\sim$, $j$ indexes the homomorphism $(\tau^{i,j}_k)_*$ in the conjugacy class $k$, $i$ indexes the coloring $\tau^{i,j}_k$ corresponding to the homomorphism $j$.

To compute $Z(Id_Y):V(Y)\longrightarrow V(Y)$ under this basis,
we triangulate $Y\times I$ so that the restriction to boundaries $Y\times{0}$ and $Y\times{1}$ is the one on $Y$. The number of vertices of the boundary of $Y\times I$ is $2v$. Let $a$ be the number of vertices of $Y\times I$. Lemma \ref{DW:Lem2} implies 
\begin{align*}
Z(Id_Y)(\tau^{i,j}_k)(\tau^{i^{\prime},j^{\prime}}_{k^{\prime}})
&=\frac{|G|^v}{|G|^a}\sharp\text{Col}(Y\times I,\tau^{i,j}_k\sqcup\tau^{i^{\prime},j^{\prime}}_{k^{\prime}})
\\
&=\left\{\begin{aligned}
\frac{1}{|G|^v}|C_G({Im(\tau^{i,j}_k)_*})|&&k=k^{\prime} \\
0&&k\neq k^{\prime}
\end{aligned}
\right.
\end{align*}
where $C_G({Im(\tau^{i,j}_k)_*})$ is the centralizer of $Im(\tau^{i,j}_k)_*$ in $G$.

It follows that the transformation matrix for $Z(Id_Y)$ is a block-diagonalized matrix with each block labelled by the conjugation class $k$ in Hom$(\pi_1(Y),G)/\sim$, and in each block, all the entries are the same number $\frac{1}{|G|^v}|C_G({Im(\tau^{i,j}_k)_*})|$. Thus we can find a basis $\{e_k\}$ for $V_G(Y)=Im(Z(Id_Y))$ by
$$e_k=\sum_{i,j}\frac{1}{|G|^v}|C_G({Im(\tau^{i,j}_k)_*})|\tau^{i,j}_k
$$
where $k$ labels the conjugation class in Hom$(\pi_1(Y),G)/\sim$.

To compute the action of $[f]\in \mcM(Y)$ on $V_G(Y)$, we define an action of $[f]$ on the set of conjugacy class of Hom$(\pi_1(Y),G)$ first. For any $[\phi]\in\text{Hom}(\pi_1(Y),G)/\sim$, $[f]([\phi])$ is defined to be $[\phi f_*^{-1}]$.

A calculation similar to the above results in  
\begin{align*}
Z(\zeta_{[f]}=(Y\times I,Y\times{0},Y\times{1},f,id))(\tau^{i,j}_k)(\tau^{i^{\prime},j^{\prime}}_{k^{\prime}})
\\
=\left\{
\begin{aligned}
\frac{1}{|G|^v}|C_G({Im(\tau^{i^{\prime},j^{\prime}}_{k^{\prime}})_*})|&&[f](k)=k^{\prime} \\
0&&[f](k)\neq k^{\prime}
\end{aligned}
\right.
\end{align*}
We obtain 
\begin{align*}
Z(\zeta_{[f]})(e_k)&=Z(\zeta_{[f]})(\sum_{i,j}\frac{1}{|G|^v}|C_G({Im(\tau^{i,j}_k)_*})|\tau^{i,j}_k)
\\
&=\sum_{i,j}\frac{1}{|G|^v}|C_G({Im(\tau^{i,j}_k)_*})|Z(\zeta_{[f]})(\tau^{i,j}_k)
\\
&=\sum_{i,j}\frac{1}{|G|^v}|C_G({Im(\tau^{i,j}_k)_*})|\sum_{i^{\prime},j^{\prime}}\frac{1}{|G|^v}|C_G({Im(\tau^{i^{\prime},j^{\prime}}_{[f](k)})_*})|\tau^{i^{\prime},j^{\prime}}_{[f](k)}
\\
&=\sum_{i^{\prime},j^{\prime}}\frac{1}{|G|^v}|C_G({Im(\tau^{i^{\prime},j^{\prime}}_{[f](k)})_*})|\frac{|G|^{v-1}}{|G|^v}[G:Z_{Im(\tau^{i,j}_k)_*}]|C_G({Im(\tau^{i,j}_k)_*})|\tau^{i^{\prime},j^{\prime}}_{k^{\prime}}
\\
&=\sum_{i^{\prime},j^{\prime}}\frac{1}{|G|^v}|C_G({Im(\tau^{i^{\prime},j^{\prime}}_{[f](k)})_*})|\tau^{i^{\prime},j^{\prime}}_{[f](k)}
\\
&=e_{[f](k)}
\end{align*}
\end{proof}
\noindent

Suppose $\rho,\rho_{DR}$ are the representations of $\mcM(Y\times \mbbS^1)$ and $\mcM(Y)$ from a TQFT $(Z,V)$ and its DR  $(Z_{DR},V_{DR})$, respectively, then the following diagram commutes:
$$\xymatrix{
    \mcM(Y)\ar[rr]^{i}\ar[dr]^{\rho_{DR}} & & \mcM(Y\times \mbbS^1) \ar[dl]^{\rho}\\
   & GL(V(Y\times \mbbS^1)),&
   }
$$
where $i$ is an inclusion: for any $[f]\in \mcM(Y)$, $i{[f]}=[f\times id]$, where $f\times id: Y\times S^1\longrightarrow Y\times \mbbS^1$ by $f\times id(x,y)=(f(x),y)$.

Apply the above diagram to DW-TQFTs, we obtain
$$\xymatrix{
    \mcM(Y)\ar[rr]^{i}\ar[dr]^{\rho_{DR}^{DW}} & & \mcM(Y\times \mbbS^1) \ar[dl]^{\rho^{DW}}\\
   & GL(\mathbb{C}\text{Hom}(\pi_1(Y\times \mbbS^1),G)/\sim),&
   }
$$
where $\sim$ means conjugacy equivalence. Then 
\begin{align*}
\text{Hom}(\pi_1(Y\times \mbbS^1),G)/\sim
&=\text{Hom}(\pi_1(Y)\times\pi_1(\mbbS^1),G)/\sim
\\
&=\{([\phi],[g])|[g]\text{ is a conjugacy class of }G,
\\
&[\phi]\in\text{Hom}(\pi_1(Y),G)/\sim\}
\\
&=\bigcup_{[g]\text{ conjugacy class}}\text{Hom}(\pi_1(Y),Z_g)/\sim
\end{align*}
\noindent
where $g$ is the image of the generator of second factor $\mathbb{Z}$ of $\pi_1(Y)\times\mathbb{Z}$.
\\
Since the image of $i$ preserves the second factor of $\pi_1(X)\times\mathbb{Z}$ with pre-composition, we obtained the following theorem.

\begin{mythm}\label{DW:ThmDRclosed}
If $Y$ is an oriented closed $(n-1)$-manifold, then there is an intertwining map: 
$$\rho_{as}: \oplus_{[g]\in [G]} V_{C_G(g)}^{\textrm{((n-1)+1)-DW}}(Y)\longrightarrow V_G^{\textrm{(n+1)-DW}}(Y\times S^1)$$
of the representations of the mapping class groups $\mcM(Y)$ and $\mcM(Y\times \mbbS^1)$ from the DW theories associated to $\{C_G(g)\}$ and $G$, respectively. The assembly map $\rho_{as}$ is an isomorphism of the two vector spaces.

More explicitly, for any $f \in \mcM(Y)$ and 
$$\oplus_{[g]\in [G]}v_{[g]} \in \oplus_{[g]\in [G]} V_{C_G(g)}^{\textrm{((n-1)+1)-DW}}(Y),$$ we have the following identity:
$$ \rho_{as}(\oplus_{[g]\in [G]}f^{((n-1)+1)}_{*}v_{[g]})=(f\times \textrm{id})^{(n+1)}_{*} \rho_{as}(\oplus_{[g]\in [G]}v_{[g]}).$$
Note that this identity does not imply that the images of the two representations are necessarily the same.

The subgroup $C_G(g)$ of $G$ is the centralizer of $g$ in $G$ and the summation is over all conjugacy classes $[G]=\{[g]\}$ of $G$.
\end{mythm}

The mapping class group of a closed manifold is not a motion group as motion groups are about diffeomorphisms that are connected to the identity, while mapping class groups are about diffeomorphisms modulo those.  The theorem is an instance of our conjecture for 
dimension reduction of DW theories from $(3+1)$ to $(2+1)$: $\textrm{DW}^{3+1}_G \cong \oplus_{[g]\in [G]} \textrm{DW}^{2+1}_{C(g)}$ for mapping class groups.

\subsection{Reduction from the $3$-torus $T^3$ to $T^2$ for $(3+1)$-DW}

In this subsection, $G=\mathbb{Z}_n$, where $n$ is any positive integer.
Those examples serve as explicit illustrations of Thm. \ref{DW:ThmDRclosed}.

First we characterize the images of representations for $Y=T^3,T^2$ in $(3+1)$- and $(2+1)$-DW TQFTs as abstract groups.  It is known that $\mcM(T^3)=SL(3,\mathbb{Z})$ and $\text{Hom}(\pi_1(T^3),\mathbb{Z}_n)/\sim \cong \text{Hom}(\mathbb{Z}^3,\mathbb{Z}_n) \cong \mathbb{Z}_n^3$. For any $(a,b,c)\in\mathbb{Z}_n^3$ and $A\in SL(3,\mathbb{Z})$, $A_{\cdot} (a,b,c)=(a,b,c)A=(a,b,c)p(A)$, where $p$ is the natural map from $SL(3,\mathbb{Z})$ to $SL(3,\mathbb{Z}_n)$ by modulo $n$. Since $p$ is surjective, it follows that $Im\rho^{(3+1)- DW}_{\mathbb{Z}_n,T^3}=SL(3,\mathbb{Z}_n)$. Similarly, $Im\rho^{(2+1)- DW}_{\mathbb{Z}_n,T^2}=SL(2,\mathbb{Z}_n)$.

Next, as images of representations from $(2+1)$ to $(3+1)$ TQFTs, it is known  that $\mcM(T^3)=SL(3,\mathbb{Z})$ is generated by $S_3,T_3$ and $\mcM(T^2)=SL(2,\mathbb{Z})$ is generated by $S_2,T_2$, where
$$T_3=\begin{pmatrix}
1&1&0\\
0&1&0\\
0&0&1
\end{pmatrix}
,
S_3=\begin{pmatrix}
0&0&1\\
1&0&0\\
0&1&0
\end{pmatrix}
,
T_2=\begin{pmatrix}
1&1\\
0&1
\end{pmatrix}
,
S_2=\begin{pmatrix}
0&-1\\
1&0
\end{pmatrix}
.
$$
The three torus $T^3$ has three $S^1$ factors, hence $SL(2,\mathbb{Z})$ can be embdeded into $SL(3,\mathbb{Z})$ in three different ways by mapping $S_2$ to $S_{21},S_{22},S_{23}$ and $T_2$ to $T_{21},T_{22},T_{23}$, where
$$T_{21}=\begin{pmatrix}
1&1&0\\
0&1&0\\
0&0&1
\end{pmatrix}
,
T_{22}=\begin{pmatrix}
1&0&1\\
0&1&0\\
0&0&1
\end{pmatrix}
,
T_{23}=\begin{pmatrix}
1&0&0\\
0&1&1\\
0&0&1
\end{pmatrix}
,
$$
$$S_{21}=\begin{pmatrix}
0&-1&0\\
1&0&0\\
0&0&1
\end{pmatrix}
,
S_{22}=\begin{pmatrix}
0&0&-1\\
0&1&0\\
1&0&0
\end{pmatrix}
,
S_{23}=\begin{pmatrix}
1&0&0\\
0&0&-1\\
0&1&0
\end{pmatrix}
.
$$
The relations $T_3=T_{21}$ and $S_3=S_{21}S_{23}$ imply that the images of the $(2+1)$ TQFT representation by dimension reduction along the first and third factors of $T^3$ generate the image of the representation of $SL(3, \mbbZ)$.

The decomposition representations of $SL(2,p)$ and $SL(3,p)$ from DW TQFTs into irreducibles can be found using character theory, where $p$ is prime.
The character table of $SL(2,p)$ and notation can be found from \cite{humphreys75}, and 
same for $SL(3,p)$ from \cite{simpson73}.  We recall the following character tables from these two references.
\begin{table}  
\begin{center}
\begin{tabular}{|l|l|l|l| p{5cm}|}
\hline  
\ &$\psi$&$\zeta_i$&$\xi_1$&$\xi_2$
\\
\hline
1&$p$&$p+1$&$\frac{1}{2}(p+1)$&$\frac{1}{2}(p+1)$
\\
\hline  
$z$&$p$&$(-1)^i(p+1)$&$\frac{1}{2}e(p+1)$&$\frac{1}{2}e(p+1)$
\\  
\hline
$a^l$&1&$\tau^{il}+\tau^{-il}$&$(-1)^l$&$(-1)^l$
\\
\hline
$b^m$&-1&0&0&0
\\
\hline
$c$&0&1&$\frac{1}{2}(1+\sqrt{ep})$&$\frac{1}{2}(1-\sqrt{ep})$
\\
\hline
$d$&0&1&$\frac{1}{2}(1-\sqrt{ep})$&$\frac{1}{2}(1+\sqrt{ep})$
\\
\hline
\end{tabular} 
\caption{$SL(2,p)$}
\end{center}  
\end{table}
In TABLE 1, $\tau$ is a primitive $(p-1)$-th root of 1 and $e=(-1)^{\frac{p-1}{2}}$. In TABLE 2, $e^{p-1}=1$, $\rho^{p-1}=1$, $\sigma^{p+1}=\rho$, $\tau^{p^2+p+1}=1$ and $\theta^3\neq1$.
\begin{table}  
\begin{center}
\begin{tabular}{|l|l|l| p{5cm}|}
\hline  
\ &$\chi_{p(p+1)}$&$\chi_{p^2+p+1}^{(i)}$
\\
\hline
$C_1^{k_1}$&$p(p+1)$&$(p^2+p+1)\omega^{ik_1}$
\\
\hline  
$C_2^{k_2}$&$p$&$(p+1)\omega^{ik_2}$
\\  
\hline
$C_3^{k_3,l_3}$&0&$\omega^{ik_3}$
\\
\hline
$C_4^{k_4}$&$p+1$&$(p+1)(e^{ik_4}+e^{-2ik_4})$
\\
\hline
$C_5^{k_5}$&1&$e^{ik_5}+e^{-2ik_5}$
\\
\hline
$C_6^{k_6,l_6,m_6}$&2&$e^{ik_6}+e^{il_6}+e^{im_6}$
\\
\hline
$C_7^{k_7}$&0&$e^{ik_7}$
\\
\hline
$C_8^{k_8}$&-1&0
\\
\hline
\end{tabular} 
\caption{$SL(3,p)$}
\end{center}  
\end{table}

\begin{mypro}

Let $\chi^{2+1}$ be the character of the representation of $SL(2,p)$, and  $\chi^{3+1}$ be the character of the representation of $SL(3,p)$ coming from DW theories, respectively.

\begin{enumerate}

\item The representation image of $\mcM(T^3)$ from $(3+1)$-DW-TQFTs is generated by the representation images from $(2+1)$-DW TQFT representations by dimensional reduction along the first and third circles of $T^3$.

\item 
$$\chi^{2+1}=2\textbf{1}_{SL(2,p)}+\psi+2\sum_{i=1}^{\frac{p-3}{2}}\zeta_i+\xi_1+\xi_2.
$$
When $p=2$, the last three summations on the right side vanish. And in this case, $\textbf{1}_{SL(2,p)}+\psi$ is the natural representation of $S_3$ on the 3-dimensional space by permutation.

\item $$\chi^{3+1}=2\chi_1+\chi_{p(p+1)}+\sum_{i=1}^{p-2}\chi_{p^2+p+1}^{(i)}.
$$
When $p=2$, the last summation vanishes.
Here 
\end{enumerate}
\end{mypro}
\begin{proof}
When $p>2$, $SL(2,p)$ has $p+4$ conjugation classes: \textbf{1},$z$,$a^l$,$b^m$,$c$,$d$,$zc$,$zd$, where
$$\textbf{1}=\begin{pmatrix}
1&0\\
0&1
\end{pmatrix}
,
z=\begin{pmatrix}
-1&0\\
0&-1
\end{pmatrix}
,
a=\begin{pmatrix}
v&0\\
0&v^{-1}
\end{pmatrix}
,
c=\begin{pmatrix}
1&1\\
0&1
\end{pmatrix}
,
d=\begin{pmatrix}
1&v\\
0&1
\end{pmatrix}
,
$$
and $b$ is the element of order $(p+1)$, which is not diagonalizable over $\mathbb{Z}_p$ and $v$ generates the multiplicative group of $\mathbb{Z}_p$ and $1\leq l\leq\frac{p-3}{2}$ and $1\leq m\leq\frac{p-1}{2}$.

For any conjugation class $[g]$ in $SL(2,p)$, 
$$\chi^{2+1}([g])=\sharp\{\text{elements in }\mathbb{Z}_p^2\text{ fixed by }g\}=p^D,
$$
where $D=\text{dimension of eigenspace of g for eigenvalue 1}$.
Thus 
\begin{align*}
\chi^{2+1}&(\textbf{1})=p^2\\
\chi^{2+1}&(c,d)=p\\
\chi^{2+1}&(z,a^l,b^m,zc,zd)=1
\end{align*}
If $\chi^{2+1}([g])=\text{RHS}([g])$, then the proof completes, which follows from straightforward check using the character table of $SL(2,p)$.

$SL(3,p)$ has $p^2+p$ conjugation classes:
$$C_1^{(k_1)}=\begin{pmatrix}
\omega^{k_1}&0&0\\
0&\omega^{k_1}&0\\
0&0&\omega^{k_1}
\end{pmatrix}
,
C_2^{(k_2)}=\begin{pmatrix}
\omega^{k_2}&0&0\\
1&\omega^{k_2}&0\\
0&0&\omega^{k_2}
\end{pmatrix}
,
$$
$$
C_3^{(k_3,l_3)}=\begin{pmatrix}
\omega^{k_3}&0&0\\
\theta^{l_3}&\omega^{k_3}&0\\
0&\theta^{l_3}&\omega^{k_3}
\end{pmatrix}
,
C_4^{(k_4)}=\begin{pmatrix}
\rho^{k_4}&0&0\\
0&\rho^{k_4}&0\\
0&0&\rho^{-2k_4}
\end{pmatrix}
,
$$
$$
C_5^{(k_5)}=\begin{pmatrix}
\rho^{k_5}&0&0\\
1&\rho^{k_5}&0\\
0&0&\rho^{-2k_5}
\end{pmatrix}
,
C_6^{(k_6,l_6,m_6)}=\begin{pmatrix}
\rho^{k_6}&0&0\\
0&\rho^{l_6}&0\\
0&0&\rho^{m_6}
\end{pmatrix}
,
$$
$$C_7^{(k_7)}=\begin{pmatrix}
\rho^{k_7}&0&0\\
0&\sigma^{-k_7}&0\\
0&0&\sigma^{-pk_7}
\end{pmatrix}
,
C_8^{(k_8)}=\begin{pmatrix}
\tau^{k_8}&0&0\\
0&\tau^{pk_8}&0\\
0&0&\tau^{p^2k_8}
\end{pmatrix}
.
$$

For any conjugation class $[g]$ in $SL(3,p)$, 
$$\chi^{3+1}([g])=\sharp\{\text{elements in }\mathbb{Z}_p^3\text{ fixed by }g\}=p^D
$$
where $D=\text{dimension of eigenspace of g for eigenvalue 1}.$
Thus 
\begin{align*}
\chi^{3+1}&(C_1^{(0)})=p^3\\
\chi^{3+1}&(C_2^{(0)})=p^2\\
\chi^{3+1}&(C_3^{(0,l_3)},C_{4,5}^{(\frac{p-1}{2})},C_6^{(k_6,l_6,p-1)},C_7^{(n(p-1))})=p\\
\chi^{3+1}&(\text{otherwise})=1
\end{align*}
It is sufficient to prove $\chi^{3+1}([g])=\text{RHS}([g])$, which follows from a direct comparison using the character tables.
\end{proof}

\subsection{DW TQFTs as $1$-extended}\label{DW:extended}

In a $1$-extended $(n+1)$-TQFT, vector spaces are assigned to 
 $n$-manifolds $Y$ when the connected components of their boundaries $\partial Y$ are labeled.  Physically, the labels are types of topological charges for some extended excitations.  When $n=2$, they are the so-called anyon types. We are interested in $n=3$, then besides pointed excitations in the three manifold $Y$, there are $1$-dimensional excitations supported around any embeded graph $\Gamma$ in $Y$.  The shape of such an extended excitation is the boundary surface $\Sigma$ of a closed neighborhood $N(\Gamma)$ of $\Gamma$.  Therefore, we need to find the label set for each of the genus=$g$ surface $\Sigma_g$, where $g=0$ corresponds to pointed excitations and $g=1$ loop excitations. 

Our description of the label set, vector spaces, and gluing formulas for  $(3+1)$-DW TQFTs is essentially a reformulation of results in \cite{morton07}. Afterwards, we describe representations of the motion groups of links in the three sphere.

\subsubsection{Label sets for DW TQFTs}

For $(2+1)$-DW theories $(Z_G,V_G)$, the shape of a pointed excitation is a circle---the boundary of the small disk neighborhood of a point.  It is known that the labels are pairs $([g],\alpha)$, where $[g]$ is a conjugacy class, and $\alpha$ an irreducible representation (irrep) of the centalizer $C_G(g)$ of $g\in G$.  This is a special case of the general description of labels for DW TQFTs: for a shape $M$ excitation of a manifold $Y$, the label set is a pair $(\rho, \alpha)$, where $\rho$ is a representation of $\pi_1(M)$ to $G$, and $\alpha$ an irrep of the centralizer of the image of $\rho(\pi_1(M))$ in $G$.  

\begin{mypro}\label{DW:labelset}
The label set for a boundary manifold $\Sigma$ consists of pairs $\{([\rho], \alpha)\}$, where $[\rho]$ is a conjugacy class of maps $\rho:\pi_1(\Sigma)\longrightarrow G$, and $\alpha$ is an irreducible representation of the centralizer $C_G(\text{Im}(\rho))$ of the image of $\rho$ in $G$ (homomorphisms {\em up to conjugation}).
\end{mypro}

There is an involution $*$ on $\{([\rho],\alpha)\}$ by $*([\rho],\alpha)=([\rho],\alpha^*)$ where $\alpha^*$ is the dual representation of $\alpha$.

As an analogue to $(2+1)$-discrete gauge theory, we will refer to a label $([\rho],[\alpha])$ for a boundary in DW theories as a pure flux if the representation $\alpha$ is trivial.

As an example, for $G=S_3$, there are $8$ types of different labels for  $(2+1)$-TQFTs as shown in TABLE 3.
\begin{table}  
\begin{center}
\begin{tabular}{|l|l|l| p{5cm}|}
\hline  
$[\rho]$&$C_G(Im(\rho))$&$\alpha_{C_{G}(Im(\rho))}$
\\
\hline
$[1]$&$S_3$&$\alpha_{S_3}^1,\alpha_{S_3}^2,\alpha_{S_3}^3,$
\\
\hline  
$[(12)]$&$\mathbb{Z}_2$&$\alpha_{\mathbb{Z}_2}^1,\alpha_{\mathbb{Z}_2}^2$
\\  
\hline
$[(123)]$&$\mathbb{Z}_3$&$\alpha_{\mathbb{Z}_3}^1,\alpha_{\mathbb{Z}_3}^2,\alpha_{\mathbb{Z}_3}^3$
\\
\hline
\end{tabular} 
\caption{$((2+1),G=S_3)$}
\end{center}  
\end{table}
where $\alpha_{S_3}^i$ are all the irreducible representations of $S_3$ up to conjugation and similarly for $\alpha_{\mathbb{Z}_2}^i,\alpha_{\mathbb{Z}_3}^i$.

For $(3+1)$-TQFTs with $G=S_3$, there are $3$ types for the two sphere as shown in TABLE 4.
\begin{table}  
\begin{center}
\begin{tabular}{|l|l|l| p{5cm}|}
\hline  
$[\rho]$&$C_G(Im(\rho))$&$\alpha_{C_{G}(Im(\rho))}$
\\
\hline
$[1]$&$S_3$&$\alpha_{S_3}^1,\alpha_{S_3}^2,\alpha_{S_3}^3,$
\\
\hline
\end{tabular} 
\caption{$((3+1),\text{ two sphere },G=S_3)$}
\end{center}  
\end{table}
And there are $21$ types for the torus.

We follow the notation of \cite{morton07} in this sub-section. For any two categories $\mcC,\mcD$, $[\mcC,\mcD]$ denotes the category of functors from $\mcC$ to $\mcD$ as objects and natural transformations as morphisms. For any compact $m$-manifold $M$ $(m\geq 0)$ with base points $b_i$ on the connected components $M_i$ $(i=1,...,k)$ of $M$, let $\Pi_1(M)$ be the category with $\{b_i\}$ as the set of objects, the fundamental group $\pi_1(M_i,b_i)$ as the morphism set of $b_i$, and the empty set as morphism space between any two different base points. 

Given any finite group $G$, we have $[\Pi_1(M),\underline{G}]$, where $\underline{G}$ is the category with only one object whose morphism set is $G$. An explicit description of $[\Pi_1(M),\underline{G}]$ is as follows. Each object is a $k$-tuples $(\rho_1,...,\rho_k)$ where $\rho_i$ is a homomorphsim from $\pi_1(M_i,b_i)$ to $G$. The morphism set between $(\rho_1^1,...,\rho_k^1),(\rho^2_1,...,\rho^2_k)$ is $\prod_iG_i,$ where $G_i=\{g\in G|gh^1=h^2g,h^j\in\text{Im}\rho_i^j,j=1,2\}$. A presentation for the skeleton of $[\Pi_1(M),\underline{G}]$ follows from this description. Each object of the skeleton is a $k$-tuples $([\rho_1],...,[\rho_k]),$ where $[\rho_i]$ is conjugacy class and the morphism set is $\prod_iC_G(\text{Im}\rho_i)$. From now on, $[\Pi_1(M),\underline{G}]$ denotes the skeleton. 

The representation category of $[\Pi_1(M),\underline{G}]$ is defined as the functor category $[[\Pi_1(M),\underline{G}],{\mcV}ec],$ where ${\mcV}ec$ is the category of finite dimensional vector spaces and linear transformations. Next we describe the irreducible representations, i.e. irreducible objects of $[[\Pi_1(M),\underline{G}],{\mcV}ec]$. For any $([\rho_1],...,[\rho_k])$ and a functor object $\alpha\in \textrm{Obj}([[\Pi_1(M),\underline{G}],{\mcV}ec])$, the image $\alpha([\rho_1],...,[\rho_k])$ of $\alpha$ is a finite dimensional vector space and $\alpha(\text{Aut}(([\rho_1],...,[\rho_k])))=\alpha(\prod_i C_G(\text{Im}\rho_i))$ defines a $\prod_i C_G(\text{Im}\rho_i)$-action on the vector space $\alpha([\rho_1],...,[\rho_k])$. By the additive structure induced from ${\mcV}ec$, the irreducible representations of $[\Pi_1(M),\underline{G}]$, i.e. the irreducible objects of $[[\Pi_1(M),\underline{G}],{\mcV}ec]$, are irreducible representations $\alpha$ of the group $\prod_i C_G(\text{Im}\rho_i)$ for some $([\rho_1],...,[\rho_k])$ and representatives $\rho_i$. Let $L_M$ be the set of irreducible representations of $[\Pi_1(M),\underline{G}]$. Specializing to the boundary closed $(n-1)$-manifold $\Sigma$, we obtain the label set $L_\Sigma$ for $\Sigma$ as in Prop. \ref{DW:labelset}.

\subsubsection{Vector spaces}

Let $Y$ be a connected $n$-manifold with boundary components $(\Sigma_i,b_i), i=1,..,k$ labeled by $\{([\rho_i], \alpha_i)\}$ and with based points $b_i$. 
We also need to choose a base point $b$ of $Y$ and a fixed collection of arcs $A_i$ in $Y$ such that $A_i$ connecting $b$ to $b_i$.  Then for any $\rho:\pi_1(Y,b)\longrightarrow G$, $A_i$ induces a homomorphism $\rho|_{\Sigma_i,b_i}:\pi_1(\Sigma_i,b_i)\longrightarrow G$ by first mapping into $\pi_1(Y,b_i)$ then using $A_i$ to send the image to $\pi_1(Y,b)$, and finally composing with $\rho$. For any $\rho$ such that $[\rho|_{\Sigma_i,b_i}]=[\rho_i]$, the centralizer $C_G(\text{Im}\rho)$ is conjugacy to one subgroup of $C_G(\text{Im}\rho_i)$ by $\tau_i$. Then an representation of $C_G(\text{Im}\rho)$ can be defined as follows. Given labels $([\rho_i],\alpha_i)$, we have a representation $\otimes_i\alpha_i$ of $\prod_iC_G(\text{Im}\rho_i)$. The homomorphism $\prod_i\tau_i:C_G(\text{Im}\rho)\longrightarrow\prod_iC_G(\text{Im}\rho_i)$
 induces a representation $\otimes\alpha_i$ of $C_G(\text{Im}\rho)$.

\begin{mydef}\label{DW:vector}
Define the vector space as 
$$V(Y;([\rho_i],\alpha_i))=\bigoplus_{[\rho]}T_{\rho},
$$
where $\rho:\pi_1(Y, b)\longrightarrow G$, and the sum is over all the conjugacy classes $[\rho]$ such that $[\rho|_{\Sigma_i,b_i}]=[\rho_i]$, and $T_{\rho}$ is the trivial component of the representation $\otimes_i\alpha_i$ of $C_G(\text{Im}\rho)$ for some representative $\rho$.
\end{mydef}

Note that first the space $T_{\rho}$ depends on the choice of representatives of the conjugacy class and the paths $A_i$. Moreover, when $\alpha_i$ are all pure fluxes, the vector space is isomorphic to $\mathbb{C}\{[\rho]|[\rho|_{\Sigma_i,b_i}]=[\rho_i]\}$.

A label $([\rho],[\alpha])$ with $\alpha$ the trivial representation is  referred to as a pure flux. For pure fluxes,  $V(Y;[\rho_i],\unit)=V(Y;[\rho_i])=\mathbb{C}\{[\rho]|[\rho|_{S_i}]=[\rho_i]\}$. 
For future computation, $V(Y;[\rho_i])$ can be reformulated as follows. Given labels $[\rho_i]$, choose a representative $\rho_i$ for each conjugacy class $[\rho_i]$, then we define the vector space $V(Y;\rho_i)$ to be spanned by $\{(\rho,a_i)|\rho:\pi_1(Y,b)\longrightarrow G,a_i(\rho|_{\Sigma_i,A_i})a_i^{-1}=\rho_i\}/\sim,$ where $(\rho,a_i)\sim(\rho^{\prime},a_i^{\prime})$ iff $[\rho]=[\rho^{\prime}]$. Then $V(Y;[\rho_i])$ is isomorphic to $V(Y;\rho_i)$.

\subsubsection{Disk, annulus, and gluing axioms} 

Suppose that $Y$ is an oriented compact $n$-manifold 
with boundary $\partial{Y}=\Sigma=\sqcup_i \Sigma_i$.
The manifold $Y$ can be regarded as a bordism between $\Sigma$ and the empty manifold. By \cite{morton07}, $Y$ defines a functor $[[\Pi_1(Y),\underline{G}],{\mcV}ec]$ from $[[\Pi_1(\Sigma),\underline{G}],{\mcV}ec]$ to $[[\Pi_1(\emptyset),\underline{G}],{\mcV}ec]={\mcV}ec$.  
The functor  $[[\Pi_1(Y),\underline{G}],{\mcV}ec]$ corresponds to a $(1\times l)$-matrix with vector spaces as entries and the columns indexed by $L_\Sigma$, where $l=|L_\Sigma|=|\prod_i L_{\Sigma_i}|$. Then we assign to any oriented compact $n$-manifold $Y$ with boundary components $\Sigma_i$ labelled by irreducible representations in $L_{\Sigma_i}$ the vector space in the matrix $[[\Pi_1(Y),\underline{G}],{\mcV}ec]$ indexed by the same irreducible representation in $L_\Sigma$. A direct computation proves that the vector space is the same as the $V(Y;([\rho_i],\alpha_i))$ in Def. \ref{DW:vector}.

Let $D$ be an $n$-ball with boundary labelled by $([\rho],\alpha)$ where $\rho:\pi_1(\partial D)\longrightarrow G$. Since $\pi_1(D)$ is trivial, to make $V(D;[\rho],\alpha)$ nonzero, $[\rho]$ has to be the trivial map. Thus $T_{\rho}$ is the irreducible representation of $G$. Since $T_{\rho}$ is the trivial component, $V(D;[\rho],\alpha)=\mathbb{C}$ when $([\rho],\alpha)=(\unit,\unit)$ and $0$ otherwise.

Let $\Sigma$ be any closed $n-1$-manifold. Label two boundaries of $\Sigma\times I$ by $([\rho_i],\alpha_i)(i=1,2)$. Since $\pi_1(\Sigma\times I)=\pi_1(\Sigma)$, to make $V(\Sigma\times I;([\rho_i]),\alpha_i)$ nonzero, $[\rho_1]=[\rho_2]$ and thus we suppose that $\rho_1=\rho_2$ and $\alpha_i$ are the irreducible representations of $C_G(Im(\rho_1))$. Then
$$T_{\rho_1}=\text{Hom}_G(1,\alpha_1\otimes\alpha_2)=\text{Hom}_G(\alpha_1^*,\alpha_2)
$$
Since $\alpha_i$ are irreducible, $T_{\rho_1}=\mathbb{C}$ when $\alpha_1^*=\alpha_2$ and $0$ otherwise.

To verify the gluing axiom, let $Y_{gl}$ be the $n$-manifold obtained from gluing two boundary components $\{\Sigma,-\Sigma\}$ of an $n$-manifold $Y$. Suppose the other boundary components $\Sigma_i$ of $Y$ are labelled by $([\rho_i],\alpha_i)$. Let $Y$ be the bordism from $\sqcup_i \Sigma_i$ to $\Sigma \sqcup- \Sigma$, $\Sigma \times I$ be the bordism from $\Sigma \sqcup -\Sigma$ to the empty manifold. Then $Y_{gl}$ is the composition of the two bordisms $Y,\Sigma \times I$, which is from $\sqcup_i \Sigma_i$ to the empty manifold, and Thm. 10 in \cite{morton07} implies the gluing axiom.

\subsubsection{Linear isomorphisms for extended bordisms}\label{DW:corner} 
Extended bordisms are bordisms between bordisms with corners.
Suppose the following diagram commutes as in \cite{morton07}:
$$\xymatrix{
\Sigma \ar[d]^{\tau_2} \ar[r]^{\tau_1} &Y_1\ar[d]^{\theta_1}\\
Y_2 \ar[r]^{\theta_2}          &X}
$$
such that $\tau_i:\Sigma \longrightarrow\partial Y_i$, $\theta_1\sqcup\theta_2:Y_1\sqcup_{\tau_1(\Sigma)=\tau_2(\Sigma)} Y_2\longrightarrow\partial X$. Then $[[X,\underline{G}],{\mcV}ec]$ defines a linear isomorphism between the vector spaces in the matrices $[[Y_i,\underline{G}],{\mcV}ec]$ with the corresponding labels.

\subsection{Pure flux representations of the motion groups of links from DW TQFTs}
From now on, we just consider pure flux labels and labels can be regarded as conjugacy classes of homomorphisms. Suppose that $L=\sqcup_iL_i$ is an oriented link in $\mathbb{S}^3,$ where $L_i$ the connected components and $N(L_i)$ a tubular neighbourhood. Any $f\in\mathcal{H}^+(\mathbb{S}^3;L)$ induces an isotopy class in $\mathcal{H}^+(\mathbb{S}^3\backslash \cup_iN(L_i),\cup_i\partial N(L_i))$. Then the following diagram commutes as in Sec. \ref{DW:corner}:
$$\xymatrix{
\partial N(L)\times\{0,1\}\ar[rr]^{id\sqcup id}\ar[d]^{id\sqcup f|_{\partial N(L)}}
& 
&(\mathbb{S}^3 \backslash N(L))\times\{0,1\}\ar[d]^{id\sqcup f}\\
\partial N(L)\times I\ar[rr]^{id}
& 
&(\mathbb{S}^3 \backslash N(L))\times I
}.
$$

Choose base points $b,b_i$ for $\mathbb{S}^3 \backslash N(L),\partial N(L_i)$, respectively. Each $L_i$ is labeled by a representation $\rho_i$, where $\rho_i:\pi_1(\partial N(L_i),b_i)\longrightarrow G$. More concretely, the representation $\rho_i$ can be parameterized by a pair of commuting group elements, so each $L_i$ is labeled by $(g_i,h_i)\in G^2$, where $g_ih_i=h_ig_i$ and $g_i,h_i$ correspond to the meridian and longitude of $\partial N(L_i)$ at $b_i$, respectively. Any motion $h_t$ in $\mcM(L\subset \mbbS^3)$ defines a linear isomorphism between $V_G(\mathbb{S}^3 \backslash N(L);\rho_i)$ and $V_G(\mathbb{S}^3 \backslash N(L);f|_{\partial N(L)}(\rho_i))$, where $f(\rho_i)$ comes from pre-composition and $f=h_1$. Then the action of $h_t\in\mathcal{M}(L\subset\mathbb{S}^3)$ is defined by $\partial (h_t)=h_1\in\mathcal{H}^+(\mathbb{S}^3;L)$ as above. When $[\rho_i]=[h_1(\rho_i)]$ for any $h_t$, we obtain a representation of $\mathcal{M}(L\subset \mathbb{S}^3)$ on the vector space $V_G(\mathbb{S}^3 \backslash N(L);\rho_i)$. 

\section{Dimension reduction conjecture for representations of motion groups of links with generalized axes}\label{motiontorus}

In this section, we focus on representations of the motion groups of links with generalized axes from $(3+1)$-DW TQFTs labeled by pure fluxes.  The main technical theorems are Goldsmith's presentation of such motion groups and dimension reduction.  We conjecture that the representations of such motion groups in general can be reduced to representations of certain surface braid groups.

\subsection{Main conjecture for representations of motion groups of links with generalized axes}
Given a link $L$ with a generalized axis $\gamma$ and fiber $F$.  
Since $F\longrightarrow \mbbS^3\backslash \gamma\longrightarrow \mbbS^1$ is the fibration and $L$ intersects with each fiber transversely in $P$, we have an induced fibration $F\backslash P\longrightarrow \mbbS^3\backslash (\gamma\cup L)\longrightarrow \mbbS^1$, and the following exact sequence
$$1\longrightarrow\pi_1(F\backslash P)\stackrel{i}{\longrightarrow}\pi_1(\mbbS^3 \backslash (\gamma\cup L))\stackrel{\pi}{\longrightarrow}\pi_1(\mbbS^1)\longrightarrow1.
$$

Choose base points $p,p_i,p_{\gamma}$ on $F\backslash P\subset\mathbb{S}^3\backslash(\gamma\cup L)$ for $\mathbb{S}^3\backslash{\gamma\cup L},L_i,\gamma$, and arcs $A_i,A_{\gamma}$ on $F\backslash P\subset\mathbb{S}^3\backslash(\gamma\cup L)$ connecting $p$ to $p_i,p_{\gamma}$. Label $L_i,\gamma$ by $[g_i,h_i],[g,h]$ such that $[g_i,h_i]$ are consistent with the action of $\mathcal{H}^+(\mathbb{S}^3,\gamma\cup L)$ on $L_i$. Next the corresponding labels for $F\backslash P$ are given as follows. The intersection points between $L_i$ and $F$ are labelled by $[g_i]$ and $\partial F$ by $[h]$. Then for any $[\rho]\in V_G(\mathbb{S}^3 \backslash (\gamma\cup L);[g,h],[g_i,h_i])$, $i^*([\rho])\in V_G(F\backslash P;[g_i],[h])$, where $i^*$ induced by $i_*:\pi_1(F\backslash P,p)\longrightarrow\pi_1(\mathbb{S}^3\backslash(\gamma\cup L),p)$. And for any $h_1\in\mathcal{H}_{\phi}(F,P)$, we have $i^*eJ(h_1)([\rho])=h_1(i^*([\rho]))$.

Our main conjecture is as follows.

\begin{conj}\label{mainconjecture}

There is a map from labels of a $1$-extended $(3+1)$-TQFTs to those of $1$-extended $(2+1)$-TQFTs so that the representations of the motion groups of a link $L$ from the $(3+1)$-TQFT decompose as direct sums of the representations of some surface braid groups of the fiber surface $F$ if the link $L$ has a generalized axis $\gamma$ with fiber $F$.

\end{conj}

Equations (\ref{equ: Dahm}) (\ref{equ:fiber}) suggest a form of the decomposition of the motion group representations or rather an organization of some surface braid group representations.  The non-trivial monodromy map of the generalized axis knot complement fibration induces a map of surface braid group representations.  The non-trivial fibfation should manifest itself in that the motion group representation should be essentially the sub-representation of a direct sum of surface braid group representations consisting of commutants of this monodromy map.

Thm. \ref{DW:ThmDRclosed} is the analogous statement of this conjecture for the mapping class groups.  In the following, we will provide further evidence by studying the motion groups of the necklace links and the torus links.  Our main Thm. \ref{DW:torus} is formulated using the fundamental group of a $2$-complex, which is not a fiber surface. But our theorem should be equivalent to one using a braid group of some fiber surface.

\subsection{Representations of the motion groups of the necklace links, and cablings of the Hopf link}

In this subsection, we obtain results for the representations of the motion groups of several families of links from DW TQFTs labeled by pure fluxes: 
the motion groups of the necklace links, which are used to study statistics of loops in physics \cite{levin14}, and the motion groups of the $n$-Hopf links.

\subsubsection{The necklace links} Let $L$ be necklace links as in Fig.\ref{fig:my_labelnecklace} where $L_i(i=1,...,n)$ are the components of the unlink and $L_c$ is the track circle that links them. Its fundamental group is 
$$\pi_1(\mbbS^3\backslash L,P)=F(x)\times F(x_1,...,x_n),
$$
where $F(x)$ is the infinite cyclic group generated by the loop for the meridian of $L_c$ and $F(x_1,...,x_n)$ is the free group generated by the loops for the meridians of $L_i$. The motion groups $\mathcal{M}(L\subset \mbbS^3)$ are generated by $\sigma_i(i=1,...,n-1)$ and $p$, where $\sigma_i$ interchanges $L_i$ and $L_{i+1}$, and $p$ permutes $L_i$ along the counterclockwise. Since the motion group permutes all $L_i$, we require the labels on $L_i$ to be the same. Thus $(g,h),(g_c,h_c)$ are used to label $L_i,L_c$ such that $gh=hg$, $g_ch_c=h_cg_c$. The images of $x, x_i$ in $G$ are also denoted by $x,x_i$. Then
$$V_G(\mbbS^3\backslash L;(g,h),(g_c,h_c))=\mathbb{C}\{[(x,x_i,a_i,a_c)]\}
$$
satisfying compatible conditions
$$x=a_cg_ca_c^{-1},x_i=a_iga_i^{-1}
,
x=a_iha_i^{-1},x_1\cdots x_n=a_ch_ca_c^{-1}
,
xx_i=x_ix.
$$

To have the above vector space nonzero, $h$ has to be a conjugacy of $g_c$. Then choosing the labels $(g_c,h_c),(g,g_c)$, we obtain the action of the motion group as follows.
$$\sigma_i(x)=x
$$
$$\sigma_i(x_i)=x_ix_{i+1}x_i^{-1},\sigma_i(x_{i+1})=x_i
$$
$$\sigma_i(x_j)=x_j,j\neq i,i+1
$$
$$p(x)=x
$$
$$p(x_i)=x_{i-1\text{ mod } n}
$$

Let $D$ be the disk bounded by $L_c$ in $\mbbS^3$. Then $L_i$ intersects with $D$ in $P_i$ transversely. We define an bijection $T$ from
\\
$V_{C_G(g_c)}(D\backslash\{P_i\};a_{c,0}^{-1}a_{i,0}ga_{i,0}^{-1}a_{c,0},h_c)$ to $V_{G}(\mathbb{S}^3\backslash L;(g,g_c),(g_c,h_c))$ by
$$T([x_i,a_i,a_c])=[g_c,x_i,a_i,a_c],
$$
where $[x_0,x_{i,0},a_{i,0},a_{c,0}]\in\{[x,x_i,a_i,a_c]\}_{(g,g_c),(g_c,h_c)}$. It induces a bijection $\tilde{T}$ on their automorphism spaces.
Let $i$ be the inclusion of $\mathcal{B}(D,n\;\text{pts})$ into $\mathcal{M}(L\subset \mathbb{S}^3)$. The above computation results in 
\begin{mypro}
The following diagram commutes:
$$\xymatrix{
    \mathcal{M}(L\subset\mathbb{S}^3)\ar[d]^{\rho^{(3+1)-DW}}& \mathcal{B}(D,n\; \textrm{pts})\ar[l]^{i} \ar[d]^{\rho^{(2+1)-DW}}\\
    Aut(V_G(\mbbS^3\backslash L;(g,g_c),(g_c,h_c)))&
    Aut(V_{C_{G}(g_c)}(D\backslash \{P_i\};a_{c,0}^{-1}a_{i,0}ga_{i,0}^{-1}a_{c,0},h_c))\ar[l]_{\tilde{T}\ \ \ \ }
   }
$$
Moreover $Im(\rho_G^{(3+1)-DW})={\tilde{T}}(Im(\rho_{C_G(g)}^{(2+1)-DW}))$.
\end{mypro}
\begin{figure}
    \centering
\begin{tikzpicture}[scale=0.7]
\draw (0,0) arc (-170:170:1 and 2);
\draw (3,0) arc (-170:170:1 and 2);
\draw (6,0) arc (-170:170:1 and 2);
\draw (-1,0.5)--(1.5,0.5);
\draw (2.5,0.5)--(4.5,0.5);
\draw (5.5,0.5)--(7.5,0.5);
\draw (8.5,0.5)--(9.5,0.5);
\draw (-1,0.5)--(-1,-3);
\draw (-1,-3)--(9.5,-3);
\draw (9.5,-3)--(9.5,0.5);
\draw (1,-1) node{$L_1$};
\draw (4,-1) node{$L_2$};
\draw (7,-1) node{$L_n$};
\filldraw (1,-1.6) circle(1pt) node[below]{$P_1$};
\filldraw (4,-1.6) circle(1pt) node[below]{$P_2$};
\filldraw (7,-1.6) circle(1pt) node[below]{$P_n$};
\draw (4,-4) node{$L_c$};
\draw (5.5,1) node{$\cdots$};
\filldraw (4,-2.5) circle(1pt) node[below]{$P$};
\end{tikzpicture}
    \caption{The necklace link}
    \label{fig:my_labelnecklace}
\end{figure}
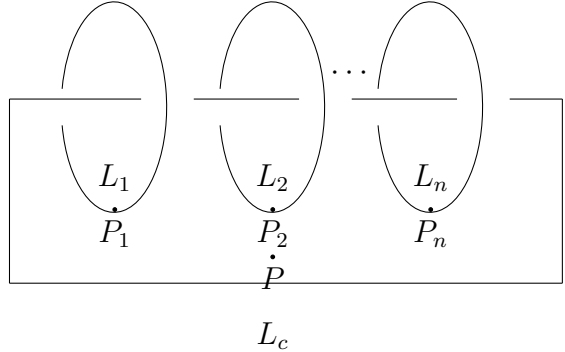

\subsubsection{The $n$-Hopf links}
Let $L$ be the link with $n$ components consisting of $n$ fibers of the Hopf fibration $\pi: \mbbS^3 \rightarrow \mbbS^2$. It is also the torus link $H_n=TL_{(1,1)^n}$. The north and south hemispheres of the base $\mbbS^2$ are the two meridian disks of the two solid tori of $\mbbS^3$. Taking one of the core circle of one solid torus as an axis, the median disk of the other solid torus is the fiber surface $F$.  

The fundamental group of the $n$-Holf link similar to Lemma 
\ref{DW:pione} is: 
$$\pi_1(\mbbS^3\backslash L)=F(y)\times F(x_1,...,x_n)
$$
and by choosing $K_\delta$ to be the disk formed by the meridian disk and $O$ as in Fig. \ref{fig:my_label3a}, labeling the components of $H_n$ by $(g,h)$ satisfying $gh=hg$, we have the following commuting diagram:
\begin{mypro}
$$\xymatrix{
    \mathcal{M}(H_n\subset \mbbS^3)\ar[d]^{\rho^{DW}}&\mathcal{B}(\text{disk},n\;
    \textrm{pts})\ar[l]^i \ar[d]^{\rho^{DW}}\\
    Aut(V_G(\mbbS^3\backslash H_n;(g,h)))&
    Aut(V_{C_{G}(h)}(K_\delta\backslash n\;\textrm{pts};\text{labeling as above}))\ar[l]^{{T_{[y]=[h]}}^{-1}\ \ \ \ \ \ \ \ \ \ \ }
   }
$$
\end{mypro}

\subsection{Representations of the motion groups of the torus links}
\begin{figure}
    \centering
   \begin{tikzpicture}[scale=0.5]
\coordinate(O) at (0,0);
\draw [->] (3,3) arc (30:60:2);
\draw (O) circle (4);
\draw (O) circle (3);
\draw (O) circle (2);
\draw (2,3) node[above right]{$S_1$};
\draw (1.5,2) node[above right]{$S_2$};
\draw (1,1) node[above right]{$S_3$};
\draw (O)--(5.22,3);
\draw (O)--(-4.35,2.5);
\draw (O)--(0,-5);
\filldraw (4.35,2.5) circle(1pt) node[right]{$P$};
\filldraw (5.22,3) circle(1pt) 
node[right]{$O^{\prime}$};
\filldraw (3.5,2) circle(1pt)
node[right]{$B^1_1$};
\filldraw (2.6,1.5) circle(1pt) node[right]{$B^1_2$};
\filldraw (1.75,1) circle(1pt) node[right]{$B^1_3$};
\filldraw (0,-4) circle(1pt)
node[below right]{$B^2_1$};
\filldraw (0,-3) circle(1pt)
node[below right]{$B^2_2$};
\filldraw (0,-2) circle(1pt)
node[below right]{$B^2_3$};
\filldraw (-3.5,2) circle(1pt) node[above]{$B^3_1$};
\filldraw (-2.6,1.5) circle(1pt) node[above]{$B^3_2$};
\filldraw (-1.75,1) circle(1pt) node[above]{$B^3_3$};
\filldraw (0,4.5) circle(1pt) node[right]{$A^1_1$};
\filldraw (0,3.5) circle(1pt) node[right]{$A^1_2$};
\filldraw (0,2.5) circle(1pt) node[right]{$A^1_3$};
\filldraw (0,1.5) circle(1pt) node[right]{$A^1_4$};
\filldraw (3.91,-2.25) circle(1pt) node[right]{$A^2_1$};
\filldraw (3.05,-1.75) circle(1pt) node[right]{$A^2_2$};
\filldraw (2.17,-1.25) circle(1pt) node[right]{$A^2_3$};
\filldraw (1.31,-0.75) circle(1pt) node[right]{$A^2_4$};
\filldraw (-3.91,-2.25) circle(1pt) node[above]{$A^3_1$};
\filldraw (-3.05,-1.75) circle(1pt) node[above]{$A^3_2$};
\filldraw (-2.17,-1.25) circle(1pt) node[above]{$A^3_3$};
\filldraw (-1.31,-0.75) circle(1pt) node[above]{$A^3_4$};
\filldraw (0,0) circle(1pt) node[right]{$O$};
\end{tikzpicture}
    \caption{$D_m$}
    \label{fig:my_label3a}
 \begin{tikzpicture}[scale=0.5]
\draw (0,0) circle (4);
\draw (0,0) circle (3);
\draw (0,0) circle (2);
\draw (-5,0)--(5,0);
\draw (3.5,3.5) node{$\tilde{S}_3$};
\draw (2,2) node[above right]{$\tilde{S}_2$};
\draw (1,1) node[above right]{$\tilde{S}_1$};
\draw[->] (3,3) arc (30:60:1);
\filldraw (5,0) circle(1pt) node[above right]{$O$};
\filldraw (4,0) circle(1pt) node[above right]{$B^1_3$};
\filldraw (3,0) circle(1pt) node[above right]{$B^1_2$};
\filldraw (2,0) circle(1pt) node[above right]{$B^1_1$};
\filldraw (1,0) circle(1pt) node[above right]{$b$};
\filldraw (-4,0) circle(1pt) node[above left]{$C^2_3$};
\filldraw (-3,0) circle(1pt) node[above left]{$C^2_2$};
\filldraw (-2,0) circle(1pt) node[above left]{$C^2_1$};
\filldraw (0,4.5) circle(1pt) node[right]{$\tilde{A}^1_4$};
\filldraw (0,3.5) circle(1pt) node[right]{$\tilde{A}^1_3$};
\filldraw (0,2.5) circle(1pt) node[right]{$\tilde{A}^1_2$};
\filldraw (0,1.5) circle(1pt) node[right]{$\tilde{A}^1_1$};
\filldraw (0,-4.5) circle(1pt) node[right]{$\tilde{A}^2_4$};
\filldraw (0,-3.5) circle(1pt) node[right]{$\tilde{A}^2_3$};
\filldraw (0,-2.5) circle(1pt) node[right]{$\tilde{A}^2_2$};
\filldraw (0,-1.5) circle(1pt) node[right]{$\tilde{A}^2_1$};
\filldraw (0,0) circle(1pt) node[above right]{$O^{\prime}$};
\end{tikzpicture}
    \caption{$D_{m^{\prime}}$}
    \label{fig:my_label4}
\end{figure}

Let $TL_{(p,q)^n}$ be the $n$-component torus link of type $(p,q)$ in $\mbbS^3$ for any coprimes $(p,q),p,q\geq 2$.  Using the parameterization of $\mbbS^3$ as in Sec. \ref{notation} or \cite{goldsmith82}, we have 
$$TL_{(p,q)^n}=\bigcup_{j=1}^n\iota(\{(\frac{1}{j}R_{\frac{qt}{p}}\Xi_p,e^{2\pi it})|0\leq t\leq1\}).
$$
Fig. \ref{fig:my_label3a} illustrates the meridian disk $D_m$ of $\mathbb{C}\times \mbbS^1$ for case $n=4$ and Fig. \ref{fig:my_label4} illustrates the meridian disk $D_{m^{\prime}}$ of $\mbbS^1\times\mathbb{C}$ for case $n=4$. The points $\{A^i_j\},\{\tilde{A}^k_j\}$ are the intersections between the $j$-th component of the torus link and the meridian disks of the first factor of $\mathbb{C}\times \mbbS^1$ and second factor of $\mbbS^1\times\mathbb{C}$, respectively, where $i=1,2,3$ and $k=1,2$.

\begin{mylem}\label{DW:pione}

Let $b$ be a base point of $\mbbS^3\backslash TL_{(3,2)^n}$, then  
$\pi_1(\mbbS^3\backslash TL_{(3,2)^n},b)$ has the following presentation:
\begin{align*}
\pi_1(\mbbS^3\backslash TL_{(3,2)^n},b)&=<x,y,u_i,v_i|x^3=y^2=u_iv_i=v_iu_i,i=1,...,n-1>
\\
&=<u_i,x,y|x^3=y^2,y^2u_i=u_iy^2>
\end{align*}
\end{mylem}
\begin{proof}
First we show that the link complement $E_L=\mbbS^3\backslash TL_{(3,2)^n}$ retracts to a 2-complex $E_L^{\prime}$.

For each $j=1,...,n-1$, a torus $T_j$ in $E_L$ is constructed as follows.
$$T_j=\iota (\{(\frac{1}{j+\frac{1}{2}}e^{i\theta},e^{i\zeta})|0\leq\theta,\zeta\leq2\pi\})
$$
Then any meridian disk $m$ of $\mathbb{C}\times S^1$ intersects $T_j$ at the circle $S_j$ as shown in Fig. \ref{fig:my_label3a}.

Next a 2-complex $K_T$ is constructed:
$$K_T=\bigcup_{j>0}\iota (\{(\frac{1}{j}R_{\frac{2t}{3}}e^{-\frac{\pi i}{3}} {\Xi}_3,e^{2\pi it})|0\leq t\leq1\})\cup x^{\prime}\cup y^{\prime}
$$
Note that $K_T$ is a 2-complex consisting of a continuous family of $(3,2)$-knots, which are parallel to $TL_{(3,2)^n}$ and form an cylinder, and the two core circles $x^{\prime},y^{\prime}$ as shown in Fig. \ref{fig:my_label5}. For any meridian disk $m$ of $\mathbb{C}\times S^1$, the intersection between $K_T$ and $m$ is just the union of the three segments pointing to $O$ as shown in Fig. \ref{fig:my_label3a}. And Fig. \ref{fig:my_label4} illustrates the intersection between $K_T$ and the meridian disk of $S^1\times\mathbb{C}$. Therefore, $K_T$ is the 2-complex obtained by attaching the two boundaries of the cylinder to the two core circles using $z\mapsto z^3$ and $z\mapsto z^2$. 

Set $E_L^{\prime}=\bigcup_{i=1}^{n-1}T_i\cup K_T.$
Retracting in each meridian disk of $\mathbb{C}\times \mbbS^1$ as shown in Fig. \ref{fig:my_label3a} gives rise to a retraction of $E_L$ to $E_L^{\prime}$.

To compute the fundamental group of $E_L^{\prime}$, we choose $b$ to be the base point as shown in Fig. \ref{fig:my_label5}, loops $a_i=\overline{bB^1_i}S_i\overline{bB^1_i}$ and $b_i=\overline{bB^1_i}\tilde{S}_i\overline{B^1_ib}$ to be generators for the fundamental group of $T_i$, and loops $x=\overline{bO}x^{\prime}\overline{OP}$ and $y=\overline{bO^{\prime}}y^{\prime}\overline{O^{\prime}b}$ to be the generators of the fundamental group of $K_T$. The direction of each $S_i$ here is chosen to be counterclockwise.

Since $K_T\cap T_i$ is a $(3,2)$-knot on $T_i$, by Van-Kampen's theorem, we arrive at
$$\pi_1(\mbbS^3\backslash T_{(2,3)^n},b)=<x,y,a_i,b_i|x^3=y^2=a_i^2b_i^3,a_ib_i=b_ia_i,i=1,...,n-1>.
$$
Setting $u_i=a_ib_i,v_i=a_ib_i^2$ leads to the desired presentation.
\end{proof}
\begin{figure}
    \centering
   \begin{tikzpicture}[scale=0.7]
\draw (0,0) ellipse [x radius=1.5,y radius=0.5];
\draw (0,3) ellipse [x radius=1.5,y radius=0.5];
\draw (0,6) ellipse [x radius=1.5,y radius=0.5];
\draw (0,5) ellipse [x radius=1.5,y radius=0.5];
\filldraw (1.5,5) circle(1pt) node[right]{$B^1_{i+1}$};
\draw (0,6.25) node[above]{$x^{\prime}$};
\draw (0,0.25) node[above]{$y^{\prime}$};
\filldraw (-0.75,3.5) circle(1pt) node[above]{$B^2_i$};
\draw (-1.5,0)--(-1.5,3);
\draw (1.5,0)--(1.5,3);
\draw (-1.5,3)--(-1.5,6);
\draw (1.5,3)--(1.5,6);
\draw (-1.5,3)--(1.5,0);
\draw (-0.75,3.5)--(1.5,6);
\draw (0,2.5)--(1.5,6);
\filldraw (1.5,3) circle(1pt) node[right]{$B^1_i$};
\filldraw (0,2.5) circle(1pt) node[above right]{$B^3_i$};
\filldraw (-1.5,3) circle(1pt) node[right]{$C^2_i$};
\filldraw (1.5,6) circle(1pt) node[right]{$O$};
\filldraw (1.5,0) circle(1pt) node[right]{$O^{\prime}$};
\filldraw (1.5,1.5) circle(1pt) node[right]{$b$};
\draw [->](0.5,7) arc (80:100:3);
\draw [->](0.5,1) arc (80:100:3);
\end{tikzpicture}
    \caption{$K_T$}
    \label{fig:my_label5}
\end{figure}

To describe the action of $\mathcal{M}(TL_{(3,2)^n}\subset \mbbS^3)$ on $\pi_1(\mbbS^3\backslash TL_{(3,2)^n},b)$, we use the presentation of  $\mathcal{M}(TL_{(3,2)^n}\subset \mbbS^3)$ in Thm. 8.7 \cite{goldsmith82}:
the generating motions for $\mathcal{M}(TL_{(3,2)^n}\subset \mbbS^3)$ are $\{\sigma_i\}_{i=1}^{n-1}$, and $\{\rho_i\}_{i=1}^n$\footnote{Our $r_i$ in Prop. \ref{presentation}}, where $\sigma_i$ interchanges the $i$-th and $(i+1)$-th components of the torus link, while $\rho_i$ rotates the $i$-th component about the  $x^{\prime}$-axis by $e^{\frac{iq\pi}{p}}$ as in Fig. \ref{fig:my_label4}.

\begin{mypro}\label{propaction}
The action of $\mathcal{M}(TL_{(3,2)^n}\subset \mathbb{S}^3)$ on $\pi_1(\mathbb{S}^3\backslash TL_{(3,2)^n},b)$ by pre-composition is as follows. Let $u_0=y,u_n=x$ formally, then 
\begin{align*}
\sigma_i(u_j)
=\left\{
\begin{aligned}
u_{i-1}u_i^{-1}u_{i+1}&&j=i
\\
u_j&&j\neq i
\end{aligned}
\right.
\end{align*}
$$\sigma_i(x)=x
$$
$$\sigma_i(y)=y
$$
\begin{align*}
\rho_i(u_j)
=\left\{
\begin{aligned}
(u_{i-1}u_i^{-1})u_j(u_{i-1}u_i^{-1})^{-1}&&j\geq i
\\
u_j&&j<i
\end{aligned}
\right.;
\end{align*}
\begin{align*}
\rho_i(x)=(u_{i-1}u_i^{-1})x(u_{i-1}u_i^{-1})^{-1}
\end{align*}
\begin{align*}
\rho_i(y)=y
\end{align*}
\end{mypro}
\begin{figure}
    \centering
    \begin{tikzpicture}[scale=0.5]
\draw (0,0) circle (3);
\draw (7,0) circle (3);
\draw (0,-7) circle (3);
\draw (7,-7) circle (3);
\draw (3.2,0)--(3.8,0)[->];
\draw (3.2,-7)--(3.8,-7)[->];
\filldraw (0,0) circle(1pt)node[right]{$O$};
\filldraw (0,2.5) circle(1pt);
\filldraw (0,2) circle(1pt);
\filldraw (2.16,-1.25) circle(1pt);
\filldraw (1.73,-1) circle(1pt);
\filldraw (-2.16,-1.25) circle(1pt);
\filldraw (-1.73,-1) circle(1pt);
\filldraw (4.5,0) circle(1pt);
\filldraw (5,0) circle(1pt);
\filldraw (7,0) circle(1pt)node[right]{$O$};
\filldraw (0,-4.5) circle(1pt);
\filldraw (0,-5) circle(1pt);
\filldraw (2.16,-8.25) circle(1pt);
\filldraw (1.73,-8) circle(1pt);
\filldraw (-2.16,-8.25) circle(1pt);
\filldraw (-1.73,-8) circle(1pt);
\filldraw (4.5,-7) circle(1pt);
\filldraw (5,-7) circle(1pt);
\filldraw (0,-7) circle(1pt)node[right]{$O$};
\filldraw (7,-7) circle(1pt)node[right]{$O$};
\draw (0,-3.2)node{$D_m$};
\draw (0,-10.2)node{$D_m$};
\draw (7,-3.2)node{$K_{\delta}$};
\draw (7,-10.2)node{$K_{\delta}$};
\draw (3.5,-3.2)node{$\sigma_i\mapsto\tilde{\sigma}_i$};
\draw (3.5,-10.2)node{$\rho_i\mapsto\tilde{\rho}_i$};
\draw [->,thick](0,2) arc (-90:75:0.25);
\draw [->,thick](0,2.5) arc (90:255:0.25);
\draw [->,thick](2.16,-1.25) arc (-30:135:0.25);
\draw [->,thick](1.73,-1) arc (150:315:0.25);
\draw [->,thick](-2.16,-1.25) arc (-150:15:0.25);
\draw [->,thick](-1.73,-1) arc (30:195:0.25);
\draw [->,thick](5,0) arc (0:165:0.25);
\draw [->,thick](4.5,0) arc (180:345:0.25);
\draw [->,thick](2.16,-8.25) arc (-30:80:2.5);
\draw [->,thick](0,-4.5) arc (90:200:2.5);
\draw [->,thick](-2.16,-8.25) arc (210:320:2.5);
\draw [->,thick](4.5,-7) arc (-180:170:2.5);
\end{tikzpicture}
    \caption{Motion of the (3,2)-torus link}
    \label{fig:my_label6}
\end{figure}
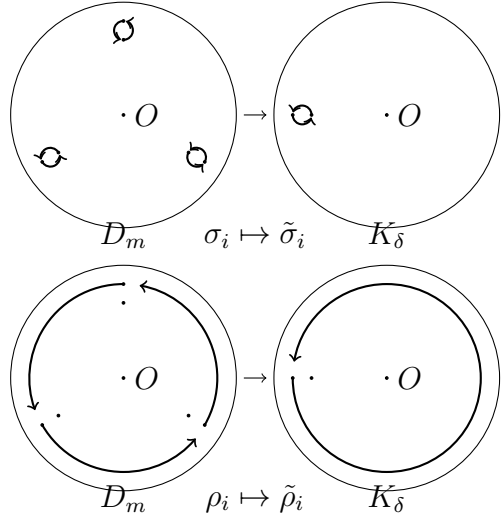
\begin{figure}
    \centering
    \begin{tikzpicture}[scale=0.5]
\draw (0,0) circle (6);
\draw (0,0) circle (4);
\draw (0,0) circle (3);
\draw (0,0) circle (2);
\draw (0,0) circle (1);
\draw (1,0)--(6,0);
\draw (1,1)node{$x^{\prime}$};
\filldraw (1,0) circle(1pt) node[above right]{$O$};
\filldraw (2,0) circle(1pt) node[above right]{$B^{1}_3$};
\filldraw (3,0) circle(1pt) node[above right]{$B^{1}_2$};
\filldraw (4,0) circle(1pt) node[above right]{$B^{1}_1$};
\filldraw (5,0) circle(1pt) node[above right]{$b$};
\filldraw (6,0) circle(1pt) node[above right]{$O^{\prime}$};
\draw (3,5)node{$y^{\prime}$};
\filldraw (-1.5,0) circle(1pt) node[above right]{$P_4$};
\filldraw (-2.5,0) circle(1pt) node[above right]{$P_3$};
\filldraw (-3.5,0) circle(1pt) node[above right]{$P_2$};
\filldraw (-4.5,0) circle(1pt) node[above right]{$P_1$};
\draw (0,5)node{$+$};
\draw (0,-5)node{$-$};
\draw (2,3) node[above right]{$K_1$};
\draw (1.5,2) node[above right]{$K_2$};
\draw (1,1) node[above right]{$K_3$};
\end{tikzpicture}
    \caption{$K_{\delta}$}
    \label{fig:my_label7}
\end{figure}
\begin{proof}
First we construct another 2-complex $K_{\delta}$ in $\mbbS^3$ as follows.
$$K_{\delta}=\bigcup_{j>0}\iota (\{(\frac{1}{j}R_te^{\frac{-\pi i}{3}}\Xi_1,e^{2\pi it})|0\leq t\leq1)\})\cup x^{\prime}\cup y^{\prime}
$$
The $2$-complex $K_{\delta}$ consists of a continuous family of $(1,1)$-knots, which form a cylinder, and the two cores $x^{\prime},y^{\prime}$. The intersection between $K_{\delta}$ and the meridian disk $D_m$ in Fig. \ref{fig:my_label3a} is just the segment $\overline{bO}$. Thus $K_\delta$ is a cylinder with $x^{\prime},y^{\prime}$ as two boundaries.

Note that the $i$-th component of $TL_{(3,2)^n}$ is on the torus $T^{\prime}_i$, where
$$T^{\prime}_i=\iota (\{\frac{1}{i}e^{i\theta},e^{i\zeta})|0\leq\theta,\zeta\leq2\pi\}).
$$

Since the intersection number between the $(1,1)$ and $(3,2)$ knots on the same torus is $1$, each of the $i$-th component of $TL_{(3,2)^n}$ intersects with $K_\delta$ exactly at $1$ point $P_i$ transversely on the torus $T^{\prime}_i$, thus $K_\delta \backslash TL_{(3,2)^n}$ can be represented by Fig. \ref{fig:my_label7}, where the circles $K_i$ are the $(1,1)$ knots on the torus $T_i$ and $P_i$ are the intersections between $TL_{(3,2)^n}$ and $K_\delta$. 

Choose $b$ to be the base point of $\mathbb{S}^3 \backslash TL_{(3,2)^n}$. 
Set $\tilde{u}_i$ to be the loops $\overline{bB^1_i}K_i\overline{B^1_ib}$. Since $K_i$ is the $(1,1)$ knot on $T_i$, it follows that  $\tilde{u}_i=a_ib_i=u_i$. Thus it suffices to describe the action of $\mathcal{M}^+(TL_{(3,2)^n}\subset \mathbb{S}^3)$ on $\tilde{u}_i,x,y$.

Consider the restriction of the action of $\sigma_i,\rho_j$ as shown in Fig. \ref{fig:my_label6} on $K_\delta$. For the interchange $\sigma_i$, we choose an isotopy $H_t$ such that each component is always on the same torus during the isotopy. Since the intersection number between $(1,1)$ and $(3,2)$ on the same torus is $1$, $\sigma_i$ induces $\tilde{\sigma}_i$ on $K_\delta$ which interchanges the $i$-th and $(i+1)$-th points as shown in Fig. \ref{fig:my_label6}. For $\rho_i$, according to Fig. \ref{fig:my_label6}, it induces $\tilde{\rho}_i$ on $K_\delta$ which rotates the $i$-th point by $2\pi$ along the direction of $K_i$. 

Next we compute the action of $\tilde{\rho_i}$ on $\tilde{u_i},x,y$. 
For simplicity, the direction of arcs is shown by $+$ and $-$ on Fig. \ref{fig:my_label7}. The path formed by the composition of arcs along points $X_1,...,X_k$ is denoted by $(X_1...X_k)$. For uniformity of notation, we set $\tilde{u}_0$ to be $y$, and $\tilde{u}_n$ to be $x$.

For any $i=1,...,n-1$, if $j\neq i$, then the motion $\tilde{\sigma}_i$ does not touch $\tilde{u}_j$. Thus  
$$\tilde{\sigma}_i(\tilde{u}_j)=\tilde{u}_j.
$$
If $j=i$, then 
\begin{align*}
\tilde{\sigma}_i(\tilde{u}_i)&=(bB^1_{i-1}+-B^1_{i-1}B^1_i-+B^1_iB^1_{i+1}+-B^1_{i+1}b)
\\
&=(bB^1_{i-1}+-B^1_{i-1}B^1_i-+B^1_ib)\cdot(bB^1_{i+1}+-B^1_{i+1}b)
\\
&=(bB^1_{i-1}+-B^1_{i-1}b)\cdot(bB^1_i-+B^1_ib)\cdot(bB^1_{i+1}+-B^1_{i+1}b)
\\
&=\tilde{u}_{i-1}\tilde{u}_i^{-1}\tilde{u}_{i+1}
\end{align*}
Since $\tilde{\sigma}_i$ does not touch $x,y$ for any $i=1,...,n-1$, hence 
$$\tilde{\sigma}_i(x)=x,\tilde{\sigma}_i(y)=y
$$
For any $i=1,...,n$, when $j<i$, $\tilde{\rho}_i$ does not touch $\tilde{u}_j$. Thus $\tilde{\rho}_i(\tilde{u}_j)=\tilde{u}_j$.
\\
If $j\geq i$, then 
\begin{align*}
\tilde{\rho}_i(\tilde{u}_j)&=
(bB^1_{i-1}+-B^1_{i-1}B^1_i-+B^1_iB^1_j+-B^1_jb)
\\
&=(bB^1_{i-1}+-B^1_{i-1}B^1_i-+B^1_iP)\cdot(PB^1_j+-B^1_jb)
\\
&\cdot(bB^1_i+-B^1_iB^1_{i-1}-+B^1_{i-1}b)
\\
&=\tilde{u}_{i-1}\tilde{u}_i^{-1}u_j(\tilde{u}_{i-1}\tilde{u}_i^{-1})^{-1}
\end{align*}
For any $i=1,...,n$, since $\rho_i$ does not touch $y$, so  $\tilde{\rho}_i(y)=y$.

For $x$, similarly 
$$\tilde{\rho}_i(x)=\tilde{u}_{i-1}\tilde{u}_i^{-1}x(\tilde{u}_{i-1}\tilde{u}_i^{-1})^{-1}.
$$
\end{proof}

By Sec. \ref{DW:extended}, to describe explicitly the representations of the motion groups of the torus links, it suffices to compute the action of the representatives of the generating motions $\{\sigma_i,\rho_i\}$ in \cite{goldsmith82}.

In the following, for easiness of notation, $u_i$ is identified with $\tilde{u}_i$. Given pure flux labels $(g,h)$ on $TL_{(3,2)^n}$ such that $g$, $h$ correspond to the meridian $m_i$ and longitude $l_i$, respectively, on the boundary, and the paths $A_i$ connecting $b$ to $P_i$ on the boundary of $\mathbb{S}^3\backslash TL_{(3,2)^n}$, then the following holds:
$$A_im_iA_i^{-1}=u_{i-1}u_i^{-1},
A_il_iA_i^{-1}=x^3.
$$
Thus for any $[(\rho,a_1,...,a_n)]\in V_G(\mathbb{S}^3\backslash TL_{(3,2)^n)};(g,h))$, we can find a representative $(x,y,u_i,a_i)$ satisfying the following conditions.  For simplicity, elements are identified with their images of $\varphi$.
$$x^3=y^2,
yu_1^{-1}=a_1ga_1^{-1},$$
$$u_1u_2^{-1}=a_2ga_2^{-1},\cdots,
u_{n-1}x^{-1}=a_nga_n^{-1},y^2=a_iha_i^{-1}.
$$
Consider the $n$-punctured cylinder $M=K_\delta \backslash \{P_i\}$ as constructed above. For any finite group $H$, labeling the boundary of $K_\delta \backslash \{P_i\}$ by $(g_i,g_x,g_y)$, we obtain 
$$V_{H}(K_\delta \backslash \{P_i\};(g_i,g_x,g_y))=\mathbb{C}\{[(\tilde{x},\tilde{y},\tilde{u}_i,\tilde{a}_i,\tilde{a}_x,\tilde{a}_y)]\},
$$
where $(\tilde{x},\tilde{y},\tilde{u_i},\tilde{a_i},\tilde{a_x},\tilde{a_y})$ satisfy the following conditions.
$$\tilde{x}=\tilde{a}_x\tilde{g}_x\tilde{a}_x^{-1}
,
\tilde{y}=\tilde{a}_y\tilde{g}_y\tilde{a}_y^{-1}
,$$
$$
\tilde{y}\tilde{u}_1^{-1}=\tilde{a}_1\tilde{g}_1\tilde{a}_1^{-1}
,
\tilde{u}_1\tilde{u}_2^{-1}=\tilde{a}_2\tilde{g}_2\tilde{a}_2^{-1}
,
\cdots
,
\tilde{u}_{n-1}\tilde{x}^{-1}=\tilde{a}_n\tilde{g}_n\tilde{a}_n^{-1}
$$
and $[-]$ denotes the conjugation class of homomorphisms.

Next we construct a bijection $\Psi_{[x],[y]}$ for $[x],[y]\in [G]$, where $[G]$ is set of the conjugation class of $G$, from
$$S_{[x],[y]}=\{[(x,y,u_i,a_i)]|x\in [x],y\in [y]\}_{\mbbS^3\backslash TL_{(2,3)^n};G;(g,h)}$$ to
$$F_{[x],[y]}=\{[(\tilde{x},\tilde{y},\tilde{u}_i,\tilde{a}_i,\tilde{a}_x,\tilde{a}_y)]\}_{K_{\delta} \backslash \{P_i\};Z(y_0^2);(a_{i,0}ga_{i,0}^{-1},x_0,y_0)},$$ where $[(x_0,y_0,u_{i,0},a_{i,0})]\in S_{[x],[y]}$ and $Z(y_0^2)$ is the centralizer $C_G(y_0^2)$.

Since $y,y_0\in [y]$, any element in $S_{[x],[y]}$ can be represented by $[(x,y_0,u_i,a_i)]$. Then we define
$$\Psi_{[x],[y]}([(x,y_0,u_i,a_i)])=([(x,y_0,u_i,a_ia_{i,0}^{-1},g_x,1)]),
$$
where $x=g_xx_0g_x^{-1}$ as $x,x_0\in [x]$.
\begin{mylem}
$\Psi_{[x],[y]}$ is well-defined and bijective.
\end{mylem}
\begin{proof}
First we show $T_{[x],[y]}$ is well-defined. 

Since $x=g_xx_0g_x^{-1}$, so  $y_0^2=x^3=g_xx_0^3g_x^{-1}=g_xy_0^2g_x^{-1}$. Thus $g_x\in Z(y_0^2)$.
Since $a_i^{-1}y_0^2a_i=a_{i,0}^{-1}y_0^2a_{i,0}$, hence  $a_{i,0}a_i^{-1}y_0^2a_ia_{i,0}^{-1}=y_0^2$. Thus $a_ia_{i,0}^{-1}\in Z(y_0^2)$.
For any $[(x^{\prime},y_0,u_i^{\prime},a_i^{\prime}]=[(x,y_0,u_i,a_i)]$, $(x^{\prime},y_0,u_i^{\prime})=g(x,y_0,u_i)g^{-1}$. Thus $g\in\Z(y_0^2)$. It follows that their images are the same.

Now we show that $T_{[x],[y]}$ is injective. For any $[(x,y_0,u_i,a_i)],[(x^{\prime},y_0,u_i^{\prime},a_i^{\prime})]$ such that
$$[(x,y_0,u_i,a_ia_{i,0}^{-1},g_x,1)]=[(x^{\prime},y_0,u_i^{\prime},a_i^{\prime}a_{i,0}^{-1},g_{x^{\prime}},1)]
$$
there exists a $g\in Z(y_0^2)$ such that $(x,y_0,u_i)=g(x^{\prime},y_0,u_i^{\prime})g^{-1}$. Thus $[(x,y_0,u_i,a_i)]=[(x^{\prime},y_0,u_i^{\prime},a_i^{\prime})]$.

Finally we show $T_{[x],[y]}$ is surjective. For any $[(x,y,u_i,a_i,a_x,a_y)]\in A_{[x],[y]}$, we have $y=a_yy_0a_y^{-1}$. Choose a representative $(x,y_0,u_i,a_i,a_x,1)$, and consider $(x,y_0,u_i,a_ia_{i,0})$, we directly check that $(x,y_0,u_i,a_ia_{i,0})$ satisfies the compatible conditions for $S_{[x],[y]}$. It follows that  $$T([(x,y_0,u_i,a_ia_{i,0})])=[(x,y_0,u_i,a_i,a_x,1)],$$ which completes the  proof.
\end{proof}

By Prop. \ref{propaction}, $S_{[x],[y]}$ is preserved by the motion group of $TL_{(3,2)^n}$ in $\mbbS^3$. Thus $\mathbb{C}S_{[x],[y]}$ is invariant under the action of the motion group from DW theory. Furthermore, the following diagram commutes. 
\begin{mypro}
\begin{enumerate}
    \item 
$$\xymatrix{
    \mathcal{M}(TL_{(3,2)^n}\subset \mbbS^3)\ar[d]^{\rho^{(3+1)-DW}}&\mathcal{B}(\textrm{cylinder}; \textrm{$n$ pts})\ar[l]^{i} \ar[d]^{\rho^{(2+1)-DW}}\\
    GL(\mathbb{C}S_{[x],[y]})&
    \; GL(V_{Z(y_0^2)}(K_{\delta} \backslash \{P_i\};\text{labeling as above}))\ar[l]_{\Psi^{-1}_{[x],[y]}\ \ \ \ \ \ \ \ \ \ \ \ \ \ \ \ \ }
   },
$$
where $i$ is the inclusion and $\rho^{DW}$ is the representation defined as above. Moreover, the images of the representations are the same: $Im(\rho_{G}^{(3+1)-DW})=\Psi^{-1}_{[x],[y]}(Im(\rho_{Z(y_0^2)}^{(2+1)-DW}))$

\item 
As representations of the motion groups from the DW TQFTs associated to groups $G$ and $C_G(y_0^2)$, respectively,
\begin{align*}
    V_G(\mbbS^3\backslash TL_{(3,2)^n};(g,h))&=\bigoplus_{[x],[y]}\mathbb{C}S_{[x],[y]}
    \\
    &=\bigoplus_{[x],[y]} V_{C_G(y_0^2)}((\text{cylinder} \backslash \textrm{$n$ pts});(a_{i,0}ga_{i,0}^{-1},x_0,y_0))
\end{align*}

\end{enumerate}
\end{mypro}
\noindent
These results can be generalized to general coprimes $(p,q),p,q\geq 3$ using a similar argument. 
\begin{mythm}\label{DW:torus}
Let $TL_{(p,q)^n}$ be the torus link of $n$ copies of the $(p,q)$-torus knot in $\mbbS^3$ with a presentation for $\pi_1(\mbbS^3\backslash TL_{(p,q)^n},b):$
$$\pi_1(\mbbS^3\backslash TL_{(p,q)^n},b)=<x,y,u_i|x^p=y^q,y^qu_i=u_iy^q, i=1,2,...,n>
.$$
Suppose the $n$ components of  $TL_{(p,q)^n}$ are labeled by $\{((g,h),\unit)\}_{i=1}^n$ such that $gh=hg$. Then the representation of motion group of $TL_{(p,q)^n}$ in $\mbbS^3$ from $(3+1)$-DW TQFT decomposes as:
$$V_G(\mbbS^3\backslash TL_{(p,q)^n};\{((g,h),\unit)\}_{i=1}^n)$$
$$=\bigoplus_{[x],[y]}V_{C_G(y_0^q)}((K_{u,v} \backslash \{b_i\}_{i=1}^n);(a_{i,0}ga_{i,0}^{-1},x_0,y_0))\footnote{$(K_{u,v} \backslash \{b_i\}_{i=1}^n)$ is not a manifold if $u,v\neq 1$, but we can define the vector space as the span of the representations of the fundamental group as in the DW TQFTs.},$$
where $u,v$ are positive integers such that $pv-qu=1$, and 
$$K_{u,v}=\bigcup_{j>0}\iota(\{(\frac{1}{j}R_{\frac{vt}{u}}e^{-\frac{i\pi}{p}}\Xi_u,e^{2\pi it})|0\leq t\leq1\})\cup x^{\prime}\cup y^{\prime}
$$
is the $2$-complex obtained by attaching the two boundaries of the cylinder to two circles by $z\mapsto z^u,z\mapsto z^v$.

\end{mythm}

%\printbibliography

\begin{thebibliography}{}
\bibitem{dahm62}
David Dahm. “A generalization of braid theory”. In: \emph{Princeton Ph. D. thesis} (1962).

\bibitem{simpson73}
William A Simpson and J Sutherland Frame. “The character tables for SL (3, q), SU (3, q 2), PSL (3, q), PSU (3,
q 2)”. In: \emph{Canadian Journal of Mathematics 25.3} (1973),
pp. 486–494.

\bibitem{humphreys75}
James E Humphreys. “Representations of SL (2, p)”. In:
\emph{The American Mathematical Monthly} 82.1 (1975), pp. 21–
39.

\bibitem{goldsmith82}
Deborah L Goldsmith. “Motion of links in the 3-sphere”.
In: \emph{Mathematica Scandinavica} (1982), pp. 167–205.

\bibitem{wakui92}
Michihisa Wakui. “On Dijkgraaf-Witten invariant for 3-manifolds”. In: \emph{Osaka Journal of Mathematics} 29.4 (1992),
pp. 675–696.

\bibitem{baez07}
John C Baez, Derek K Wise, Alissa S Crans, et al. “Exotic statistics for strings in 4d BF theory”. In: \emph{Advances in
Theoretical and Mathematical Physics} 11.5 (2007), pp. 707–
749.

\bibitem{morton07}
Jeffrey Morton. “Extended TQFT’s and Quantum Gravity”. In: \emph{Thesis (Ph.D.)–University of California, Riverside.} 2007. 255 pp. ISBN 978-0549-08988-9 (2007).

\bibitem{levin14}
Chenjie Wang and Michael Levin. “Braiding statistics of loop excitations in three dimensions”. In: \emph{Physical review
letters} 113.8 (2014), p. 080403.

\bibitem{LW20}L.~Mueller, and L.~ Woike. \emph{Dimensional Reduction, Extended Topological Field Theories and Orbifoldization.} arXiv preprint arXiv:2004.04689 (2020).

\bibitem{kadar17}
Zolt{\'a}n K{\'a}d{\'a}r et al. “Local representations of the loop braid group”. In: \emph{Glasgow Mathematical Journal} 59.2 (2017), pp. 359–378.

\bibitem{reutter18}
Christopher L Douglas and David J Reutter. “Fusion 2-
categories and a state-sum invariant for 4-manifolds”. In:
\emph{arXiv preprint arXiv:1812.11933} (2018).

\bibitem{rowell18}
Eric Rowell and Zhenghan Wang. “Mathematics of topological quantum computing”. In: \emph{Bulletin of the American
Mathematical Society} 55.2 (2018), pp. 183–238.

\bibitem{rowell19}
Alex Bullivant et al. “Representations of the necklace braid
group: topological and combinatorial approaches”. In: \emph{Communications in Mathematical Physics} (2019), pp. 1–25

\bibitem{Turaev}V.~G.~Turaev, "Quantum invariants of knots and 3-manifolds", Vol. 18. Walter de Gruyter Co., 2020.

\bibitem{Walker1991}K.~Walker.  "On Witten's 3-manifold Invariants", 1991 TQFT notes. 

\bibitem{walker13}
Kevin Walker and Zhenghan Wang. “(3+ 1)-TQFTs and topological insulators”. In: \emph{Frontiers of Physics} 7.2 (2012),
pp. 150–159.

\bibitem{Wang}Z.~Wang. "Topological quantum computation", No. 112. American Mathematical Soc., 2010.

\end{thebibliography}
\end{document}